\documentclass[11pt,letterpaper]{amsart}
\usepackage{amsmath}
\usepackage{amsthm}
\usepackage{amssymb}
\usepackage{amsfonts,mathrsfs}
\usepackage{stmaryrd}
\usepackage{amsxtra}  
\usepackage{epsfig}
\usepackage{verbatim}
\usepackage[all]{xy}
\usepackage{comment}

\usepackage{dsfont}
\usepackage{color}
\usepackage{mathtools}
\usepackage{tikz}
\usepackage{tikz-cd}
\usetikzlibrary{fpu}
\usetikzlibrary{matrix}

\usepackage{esint}

\theoremstyle{plain}
\newtheorem{theorem}[equation]{Theorem}
\newtheorem{proposition}[equation]{Proposition}
\newtheorem{lemma}[equation]{Lemma}
\newtheorem{corollary}[equation]{Corollary}

\newcommand{\red}[1]{{\color{red}{#1}}} 

\usepackage{bm}

\theoremstyle{remark}
\newtheorem{remark}[equation]{Remark}

\numberwithin{equation}{section}

\newcommand{\cut}{\setminus}

\usepackage{enumitem}

\DeclareMathOperator{\sinc}{sinc}
\renewcommand{\hat}{\widehat}

\newcommand{\wt}{\widetilde}
\newcommand{\ol}{\overline}

\usepackage{scalerel,stackengine}
\stackMath
\newcommand\reallywidehat[1]{%
\savestack{\tmpbox}{\stretchto{%
  \scaleto{%
    \scalerel*[\widthof{\ensuremath{#1}}]{\kern.1pt\mathchar"0362\kern.1pt}%
    {\rule{0ex}{\textheight}}
  }{\textheight}%
}{2.4ex}}%
\stackon[-6.9pt]{#1}{\tmpbox}%
}
\DeclareMathOperator{\re}{Re}

\def\norm#1{\left\Vert#1\right\Vert}

\def\<<>>#1#2{\!\left\langle\!\left\langle#1,#2\right\rangle\!\right\rangle\!}

\let\originalleft\left
\let\originalright\right
\renewcommand{\left}{\mathopen{}\mathclose\bgroup\originalleft}
\renewcommand{\right}{\aftergroup\egroup\originalright}

\newcommand{\ce}{{\mathcal E}}

\newcommand{\ch}{{\mathcal H}}
\newcommand{\ci}{{\mathcal I}}

\newcommand{\co}{{\mathcal O}}

\newcommand{\cs}{{\mathcal S}}

\newcommand{\sB}{{\mathscr B}}

\newcommand{\C}{{\mathbb C}}
\newcommand{\D}{{\mathbb D}}

\newcommand{\R}{{\mathbb R}}

\begin{document}

\title[Cauchy and Szeg\H{o} in dual Hardy spaces]{Cauchy transforms and Szeg\H{o} projections in dual Hardy spaces: inequalities and  Möbius invariance}
\author{David E. Barrett \& Luke D. Edholm}
\begin{abstract}
Dual pairs of interior and exterior  Hardy spaces associated to a simple closed Lipschitz planar curve are considered, leading to a Möbius invariant function bounding the norm of the Cauchy transform $\bm{C}$ from below. 
This function is shown to satisfy strong rigidity properties and is closely connected via the Berezin transform to the square of the Kerzman-Stein operator.
Explicit example calculations are presented. For ellipses, a new asymptotically sharp lower bound on the norm of $\bm{C}$ is produced.
\end{abstract}
\thanks{The first author was supported in part by NSF grant number DMS-1500142}
\thanks{The second author was supported in part by Austrian Science Fund (FWF) grants: DOI 10.55776/I4557 and DOI 10.55776/P36884.}
\address{Department of Mathematics\\University of Michigan, Ann Arbor, MI, USA}
\email{barrett@umich.edu}
\address{Department of Mathematics\\Universit\"at Wien, Vienna, Austria}
\email{luke.david.edholm@univie.ac.at}

\maketitle


\section{Introduction}\label{S:Intro}

Let $\gamma$ be a simple closed oriented Lipschitz curve in the Riemann sphere $\hat\C$ bounding a domain $\Omega_+$ to the left and $\Omega_-$ to the right. 
Each domain admits a Hardy space, denoted respectively by $\ch_+^2(\gamma)$ and $\ch_-^2(\gamma)$, consisting of holomorphic functions with square integrable boundary values. 
(Precise definitions are given in Section \ref{SS:curves-in-Riemann-sphere}.)
In this paper we investigate the interaction between the Hardy spaces on $\Omega_+$ and $\Omega_-$ and use our findings to deduce norm estimates and prove invariance and rigidity theorems related to two classical projection operators: the Szeg\H{o} projection, $\bm{S}$, and the Cauchy transform, $\bm{C}$.

These operators and the connections between them are well studied.
Of particular importance is a breakthrough made by Kerzman and Stein in \cite{KerSte78a} where it was shown that for smooth $\gamma$, their eponymously named operator $\bm{A} := \bm{C} - \bm{C}^*$ is compact.
This observation led to the formula $\bm{C} = \bm{S}(I+\bm{A})$ relating the Cauchy and Szeg\H{o} projections, making it possible to use known information about $\bm{S}$ to study $\bm{C}$, and vice versa.
The Kerzman-Stein operator established an alternative foundation upon which both Hardy space theory and much of classical complex analysis could be developed; this is the theme of Bell's book \cite{Bell16}.

One aim of the present paper is to investigate the function
\begin{equation}\label{E:Cauchy-Szego-ratio-intro}
z \mapsto \left( \int_\gamma |C(z,\zeta)|^2\,d\sigma(\zeta) \right)^{\frac{1}{2}}
\left( \int_\gamma |S(z,\zeta)|^2\,d\sigma(\zeta) \right)^{-\frac{1}{2}},
\tag{$*$}
\end{equation}
where $C$ is the Cauchy kernel, $S$ is the Szeg\H{o} kernel and $\sigma$ is arc length measure.
This function has a number of remarkable properties and, unsurprisingly, encodes detailed information about the Cauchy transform, the Szeg\H{o} projection and how the two operators interact.
A close relationship between $(*)$ and $\bm{A}$ (or more precisely $\bm{A} \circ \bm{A}$) via the Berezin transform is shown to hold (see Proposition \ref{P:concat-Berezin-comps}), and there are situations (e.g. Theorem \ref{C:vanishing-of-the-Kerzman-Stein-kernel}) where direct analysis of $(*)$ recaptures and even strengthens results previously obtained from the Kerzman-Stein operator.

The following theme pervades the paper:
it is natural and informative to consider the pieces of $(*)$ on $\Omega_+$ and $\Omega_-$ together as a single object.
With this in mind, several pairs of objects associated to a simple closed Lipschitz curve $\gamma$ will be considered in tandem:
\begin{enumerate}
\item {\em Domains and function spaces.} 
The interior and exterior domains $\Omega_+$ and $\Omega_-$, along with the associated Hardy spaces $\ch_+^2(\gamma)$ and $\ch_-^2(\gamma)$.

\item {\em Projection operators and kernel functions.} 
The Cauchy transforms ($\bm{C}_+$ and $\bm{C}_-$) and Szeg\H{o} projections ($\bm{S}_+$ and $\bm{S}_-$) on the interior and exterior domains, together with the representative kernel functions ($C_+$, $C_-$, $S_+$ and $S_-$).

\item {\em Two pairings of functions in $L^2(\gamma)$.} 
The usual inner product $\langle f, g \rangle$, along with a $\C$-bilinear pairing $\<<>>{f}{g}$ defined in  \eqref{E:bilinear-pairing-def} below.
The second pairing yields an alternative characterization of the Hardy dual spaces that underlies much of the theory we develop.
\end{enumerate}

\subsection{Interior and exterior projections}
Throughout the paper, many arguments can be carried out simultaneously in the interior ($\Omega_+$, $\ch_+^2(\gamma)$, $\bm{S}_+$, etc.) and exterior ($\Omega_-$, $\ch_-^2(\gamma)$, $\bm{S}_-$, etc.) settings.
Whenever possible our notation will reflect this, as we now demonstrate.

One way to construct holomorphic functions on $\Omega_+$ with $L^2(\gamma)$ boundary values is the Szeg\H{o} projection $\bm{S}_+$, the orthogonal projection from $L^2(\gamma)$ onto its holomorphic subspace $\ch_+^2(\gamma)$.
Given $h\in L^2(\gamma)$,
\begin{equation}\label{E:Szego-proj-def}
\bm{S}_+ h(z) = \int_{\gamma} S_+(z,\zeta)h(\zeta)\,d\sigma(\zeta), \qquad z \in \Omega_+,
\end{equation}
where $S_+(z,\zeta)$ is the Szeg\H{o} kernel of $\ch_+^2(\gamma)$ and $d\sigma$ is arc length measure.
This kernel is conjugate symmetric, i.e., $\ol{S_+(z,\zeta)} = S_+(\zeta,z)$, and for fixed $z \in \Omega_+$, $S_+(\cdot,z) \in \ch_+^2(\gamma)$.
Since $\bm{S}_+$ is an orthogonal projection onto $\ch_+^2(\gamma)$, we immediately obtain the reproducing property that $\bm{S}_+f = f$ for $f \in \ch_+^2(\gamma)$, as well as the fact that $\bm{S}_+^* = \bm{S}_+$.

There is a corresponding Szeg\H{o} projection $\bm{S}_-$ from $L^2(\gamma)$ onto $\ch_-^2(\gamma)$ given by a formula à la \eqref{E:Szego-proj-def}, but now using $S_-(z,\zeta)$, the Szeg\H{o} kernel of $\ch_-^2(\gamma)$, as the representative kernel.
The same basic properties of $\bm{S}_+$ and $S_+$ mentioned above hold for $\bm{S}_-$ and $S_-$, though in general the kernel functions $S_+$ and $S_-$ themselves bear no obvious resemblance.

When we meet situations as described above, where parallel facts hold in the interior and exterior settings, the presentation will be streamlined as follows:
\begin{quotation}
{\it ``The Szeg\H{o} projection $\bm{S}_\pm$ is an orthogonal projection from $L^2(\gamma)$ onto $\ch^2_\pm(\gamma)$."} 
\end{quotation}
is a condensed way of writing two statements at once.
The original string is meant to be read {\em exactly twice}, once using only the {\em top} signs, and once using only the {\em bottom} signs:
\begin{itemize}
\item {\it The Szeg\H{o} projection $\bm{S}_+$ is an orthogonal projection from $L^2(\gamma)$ onto $\ch^2_+(\gamma)$.}
\item {\it The Szeg\H{o} projection $\bm{S}_-$ is an orthogonal projection from $L^2(\gamma)$ onto $\ch^2_-(\gamma)$.}
\end{itemize}

A second way to construct holomorphic functions from $L^2$ boundary data is the Cauchy transform.
Let $\gamma$ be a simple closed Lipschitz curve in the plane and let $T$ be the (a.e. defined) unit tangent vector pointing in the counterclockwise direction.
Given $h \in L^2(\gamma)$, interior and exterior holomorphic functions $\bm{C}_\pm h \in \co(\Omega_\pm)$ are generated via the Cauchy integral
\begin{subequations}
\begin{align}
\bm{C}_{\pm}h(z) 
&= \frac{1}{2 \pi i} \oint_\gamma \frac{h(\zeta)}{\zeta-z}\,d\zeta, = \int_{\gamma} C_\pm(z,\zeta) h(\zeta) \,d\sigma(\zeta), \qquad z\in \Omega_\pm, \label{E:Cauchy-integral-formula}
\end{align}
where, upon noting that $d\zeta = \pm T(\zeta) \,d\sigma(\zeta)$, the Cauchy kernel is defined as
\begin{equation}\label{E:Cauchy-kernel-def}
C_\pm(z,\zeta) =  \frac{\pm T(\zeta)}{2\pi i(\zeta-z)}.
\end{equation}
\end{subequations}
The choice of $\pm$ specifies orientation so that holomorphic functions are reproduced.

When $z \in \gamma$, the integral \eqref{E:Cauchy-integral-formula} no longer converges in the ordinary sense.
But if non-tangential limits (see Section \ref{SS:curves-in-Riemann-sphere}) of the holomorphic function $\bm{C}_\pm h \in \co(\Omega_\pm)$ are taken, we obtain the following principle value integral for a.e. $z \in \gamma$:
\begin{equation}\label{E:Cauchy-transform-boundary}
\bm{C}_\pm h(z) = \frac{h(z)}{2} + \frac{1}{2} \,\mathrm{P.V.} \int_{\gamma} C_\pm(z,\zeta) h(\zeta)\, d\sigma(\zeta).
\end{equation}

The notion of a principle value -- where the integral is calculated over the curve with  a small {\em symmetric} portion of $\gamma$ about $z$ excised, and a limit is taken as the endpoints of the excision are sent to $z$ at the {\em same rate} -- makes sense when $\gamma$ is a $C^1$ curve and $h \in C^1(\gamma)$.
But the scope of this notion extends to a wider setting thanks to a deep result of Coifman, McIntosh and Meyer \cite{CMM82}, which says that when $\gamma$ is a Lipschitz curve, the principle value integral in \eqref{E:Cauchy-transform-boundary} both exists for almost every $z \in \gamma$ and defines a bounded operator on $L^p(\gamma,\sigma)$, $1<p<\infty$.


\subsection{Duals of Hardy spaces}

Let $\gamma$ be a simple closed oriented Lipschitz curve.
Consider two related pairings of $f,g \in L^2(\gamma)$:
the usual inner product $\langle\cdot,\cdot\rangle$ and a ($\C$-)bilinear pairing $\<<>>{\cdot}{\cdot}$ given by
\begin{equation}\label{E:bilinear-pairing-def}
\langle f, g \rangle = \int_{\gamma} f(\zeta)\ol{g(\zeta)}\,d\sigma(\zeta), \qquad\qquad \<<>>{f}{g} = \oint_{\gamma} f(\zeta)g(\zeta)\,d\zeta.
\end{equation}
Since $d\zeta = T(\zeta)\,d\sigma(\zeta)$, these pairings are related by $\langle f, g \rangle = \<<>>{f}{\ol{g T}}$ and $\<<>>{f}{g} = \langle f, \ol{g T} \rangle$, where $T$ is the unit tangent agreeing with the orientation of $\gamma$.

Since $\ch_\pm^2(\gamma)$ is a Hilbert space, the inner product $\langle\cdot,\cdot\rangle$ facilitates the canonical isometric duality self-identification $\ch_\pm^2(\gamma)' \cong \ch_\pm^2(\gamma)$.
The bilinear pairing $\<<>>{\cdot}{\cdot}$ facilitates a quasi-isometric dual space identification of the interior and exterior Hardy spaces:
\begin{equation}\label{E:dual-space-identification}
\ch_\pm^2(\gamma)' \simeq \ch_\mp^2(\gamma),
\end{equation}
see Section \ref{SS:Two-pairings-dual-spaces}, and in particular, Proposition \ref{P:interior-exterior-dual-identification}.

\subsection{The Cauchy-Szeg\H{o} $\Lambda$-function}\label{SS:Intro-section-defining-Lambda}
Let $\gamma$ be a simple closed bounded Lipschitz curve oriented counterclockwise in the plane.
Define two real-valued functions
\begin{subequations}
\begin{align*}
&\Lambda_+(\gamma,z) = \frac{\norm{C_+(z,\cdot)}_{L^2(\gamma)}}{\sqrt{S_+(z,z)}}, \qquad z \in \Omega_+, \\
&\Lambda_-(\gamma,z) = \frac{\norm{C_-(z,\cdot)}_{L^2(\gamma)}}{\sqrt{S_-(z,z)}}, \qquad z \in \Omega_- \backslash \{ \infty \}.
\end{align*}
\end{subequations}

Now combine them to form the {\em Cauchy-Szeg\H{o} $\Lambda$-function}, a real-valued function defined on the Riemann sphere by:
\begin{equation}\label{E:Lambda-def}
\Lambda(\gamma,z) 
= \begin{dcases}
\Lambda_\pm(\gamma,z), \qquad&z \in \Omega_\pm \backslash \{\infty\}, \\
1, \qquad &z \in \gamma, \\
\sqrt{\tfrac{\sigma(\gamma)}{2\pi \kappa(\gamma)}}, \qquad &z = \infty,
\end{dcases}
\end{equation}
where $\sigma(\gamma)$ denotes arc length and $\kappa(\gamma)$ denotes analytic capacity (see Section \ref{SS:Analytic-Capacity}).

\subsubsection{Basic properties}\label{SSS:basic-properties}
The assigned values for $z \in \gamma$ and $z = \infty$ are very natural:
\begin{theorem}\label{T:intro-3-props-of-Lambda}
Let $\gamma$ be a simple closed Lipschitz curve in the plane.
Then
\begin{enumerate}
\item $\Lambda(\gamma,z)$ is continuous as $z \to \infty$. \label{I:Lambda-cont-at-infty}
\item If $\gamma$ is $C^1$ smooth and $\zeta_0 \in \gamma$, then $\Lambda(\gamma,z)$ is continuous as $z \to \zeta_0$. \label{II:Lambda-cont-at-curve}
\item If $\Phi$ is a Möbius transformation with pole off of $\gamma$, then $\Lambda(\gamma,z) = \Lambda(\Phi(\gamma),\Phi(z))$. \label{III:Möbius-invariance}
\end{enumerate}
\end{theorem}

Part \eqref{I:Lambda-cont-at-infty} is proved in Theorem \ref{T:Lambda-at-infinity} after a short discussion of analytic capacity.
Part \eqref{II:Lambda-cont-at-curve} is proved in Theorem \ref{T:Lambda-continuinty-at-boundary}, with the Berezin transform and compactness of the Kerzman-Stein operator playing important roles.
These first two parts together show that $z \mapsto \Lambda(\gamma,z)$ is continuous on the Riemann sphere whenever $\gamma$ is a $C^1$ smooth curve.
Part \eqref{III:Möbius-invariance} is shown in Theorem \ref{T:Lambda-Proj-Inv} after obtaining a Möbius transformation rule for the Cauchy kernel.

One consequence of Möbius invariance is that it gives a simple way to extend $\Lambda$ to {\em unbounded} curves:
let $\gamma$ be a simple closed Lipschitz curve in the Riemann sphere passing through $\infty$ (see Section 2.1), and $\Phi$ be a Möbius transformation with its pole lying off of $\gamma$.
Then the image curve, denoted $\Phi(\gamma)$, is a simple closed Lipschitz curve in the plane, and we {\em define}
\begin{equation*}
\Lambda(\gamma,z) := \Lambda(\Phi(\gamma),\Phi(z)).
\end{equation*}

The fact this extension is well-defined is immediate from Theorem \ref{T:intro-3-props-of-Lambda}, part \eqref{III:Möbius-invariance}.

The next result shows that circles form the class of minimizing curves for $\Lambda$:
\begin{theorem}\label{T:rigidity-intro}
Let $\gamma$ be a simple closed Lipschitz curve in the Riemann sphere.
\begin{enumerate}
\item $\Lambda(\gamma,z) \ge 1$, for all $z \in \Omega_\pm$.
\item If there is a single $z \in \Omega_\pm$ such that $\Lambda(\gamma,z) = 1$, then $\Lambda(\gamma,\cdot) \equiv 1$ and $\gamma$ is a circle (or a line, including the point at $\infty$).
\end{enumerate}
\end{theorem}

This theorem and its consequences are presented in Sections \ref{SS:circles-as-extremal-curves} and \ref{SS:rigidity}.

In \cite{KerSte78a} Kerzman and Stein gave a clever geometric interpretation of their operator $\bm{A}$ and deduced that if the Cauchy and Szeg\H{o} kernels of a bounded domain coincide, the underlying domain must be a disc.
Theorem \ref{T:rigidity-intro} implies a significantly strengthened version of this result.
The proof of the following result (see Corollary \ref{C:vanishing-of-the-Kerzman-Stein-kernel}) uses the $\Lambda$-function and makes no reference to the geometry of the Kerzman-Stein operator:

\begin{corollary}\label{C:vanishing-of-the-KS-intro}
Let $\Omega$ be a bounded simply connected planar domain with Lipschitz boundary.
If there exists a single $z \in \Omega$ such that the Cauchy and Szeg\H{o} kernels satisfy 
\begin{equation*}
|C(z,\zeta)| \le |S(z,\zeta)|
\end{equation*} 
for almost every $\zeta \in \gamma$, then $\Omega$ is a disc.
\end{corollary}

\subsubsection{Estimating Cauchy norms}
The maximum value attained by $\Lambda(\gamma,\cdot)$ on the Riemann sphere bounds the norm of the Cauchy transform from below:

\begin{theorem}\label{T:Cauchy-norm-lower-bound}
Let $\gamma$ be a simple closed Lipschitz curve in the Riemann sphere.
The norms of the interior and exterior Cauchy transforms are equal and further, satisfy the estimate
\begin{equation}\label{E:Cauchy-norm-lower-bound}
\sup\nolimits_{z \in \hat{\C}} \Lambda(\gamma,z) \le \norm{\bm{C}_\pm}.
\end{equation}
\end{theorem}
\begin{proof}
That $\norm{\bm{C}_+} = \norm{\bm{C}_-}$ is shown in Theorem \ref{T:norm-pairing-efficieny}.
In Theorem \ref{T:Lambda-function-lower-bound} it is shown that $\Lambda(\gamma,z) \le \norm{\bm{C}_\pm}$ for every $z \in \C$.
The continuity of $\Lambda(\gamma,\cdot)$ at $\infty$ finishes the proof.
\end{proof}

Let $W_\theta = \{r e^{i\varphi} : r>0, \,\, |\varphi|<\theta \}$ be the unbounded wedge with aperture $2\theta \in (0,2\pi)$ and boundary denoted by $bW_\theta$.
In Section \ref{SS:the-wedge} we study this wedge and produce an explicit formula for $\Lambda(bW_\theta,z)$ in Theorem \ref{T:Lambda-wedge-formula}.
Several conclusions are drawn from this formula;
in particular, it is shown that $\Lambda(bW_\theta,z)$ is discontinuous at the origin (a corner point), breaking from the continuous behavior on $C^1$ curves guaranteed by Theorem \ref{T:intro-3-props-of-Lambda}.

In Section \ref{SS:the-ellipse}, a second family of curves is considered.
Let $\mathcal{E}_r = \{ (x,y): \frac{x^2}{r^2}+y^2 = 1 \}$, an ellipse with major-to-minor axis ratio $r>1$.
We compute $\Lambda(\ce_r,z)$ and use it to produce the best known lower estimate on the norm of the Cauchy transform.
\begin{theorem}\label{T:Cauchy-ellipse-lower-bound-intro}
Let $r>1$. The $L^2$-norm of the Cauchy transform on $\ce_r$ satisfies
\begin{equation}\label{E:Cauchy-ellipse-lower-bound}
\norm{\bm{C}}_{L^2(\ce_r)} \ge \sqrt{\frac{2}{\pi} \sqrt{1-\tfrac{1}{r^2}} \cdot \frac{ (r^2+1) \cdot \Pi\Big(1-r^2, \sqrt{1-\tfrac{1}{r^2}}\Big) - K\Big(\sqrt{1-\tfrac{1}{r^2}}\Big) }{\vartheta_2\Big(0,\big(\frac{r-1}{r+1}\big)^2 \Big) \, \vartheta_3\Big(0,\big(\frac{r-1}{r+1}\big)^2 \Big)}}.
\end{equation}
\end{theorem}
This bound is shown to be asymptotically sharp as $r \to 1$. (See Section \ref{SSS:families-of-special-functions} for conventions regarding elliptic integrals and theta functions appearing in the formula.)

\subsection{Motivation from higher dimensions}

This paper grew out of an ongoing project on the Leray transform $\bm{L}$, a higher dimensional analogue of the Cauchy transform.
Given a $\C$-convex hypersurface $\cs \subset \C\mathbb{P}^n$, recent work of the authors (see \cite{Bar16,BarEdh17,BarEdh21}) uncovers an intriguing connection between analytic quantities tied to $\bm{L}$ (norms, essential norms, spectral data) and projective-geometric invariants associated to $\cs$ and its projective dual hypersurface $\cs^*$.
A natural construction yields a pair of projectively-invariant dual Hardy spaces on $\cs$ and $\cs^*$, and a generalized version of  $\Lambda(\gamma,\cdot)$ can be defined using Leray and Szeg\H{o} kernels.
The higher dimensional theory simplifies considerably in one dimension, serving to motivate the present paper.

The function $\Lambda$ can be related to Fredholm eigenvalue problems studied by Bergman-Schiffer \cite{BergmanSchiffer1951} and Singh \cite{Singh1960}. 
Burbea previously connected the Kerzman-Stein operator $\bm{A}$ to these same eigenvalue problems in \cite{Burbea1982}, then went on to reprove key properties of $\bm{A}$ (e.g. compactness) using the theory of Garabedean anti-symmetric $l$ kernels.
Similarly, some basic properties of $\Lambda$ in Section \ref{SSS:basic-properties} can be obtained using the same approach -- at least when $\gamma$ is smooth enough.
But here we have opted to avoid the Garabedean machinery entirely.
The reason for this is two-fold. 
Firstly, in minimally smooth settings ($\gamma$ being $C^1$ or less), analysis becomes significantly harder and the Garabedean approach is often untenable.
For example, while the compactness of $\bm{A}$ continues to hold when $\gamma$ is only assumed to be $C^1$, Burbea's argument breaks down and the proof requires much more delicacy; see \cite{Lan99}.
Secondly, the theory of the Garabedean kernel depends critically on a particular orthogonal decomposition of $L^2(\gamma)$ (see \cite[Theorem 4.3]{Bell16}), one that no longer holds for $L^2(\cs)$.
As we are motivated by the higher dimensional problem, several of our proofs have been written so as to mirror that setting.

\section{Interior and exterior Hardy spaces}\label{S:Pairing-Hardy-Spaces}

\subsection{Lipschitz curves}\label{SS:curves-in-Riemann-sphere}

A function $\varphi:\R \to \R$ is called Lipschitz if there exists a constant $K >0$ (the Lipschitz constant) so that $|\varphi(x_1)-\varphi(x_2)| \le K|x_1-x_2|$ for all $x_1,x_2 \in \R$.
Such a function is differentiable almost everywhere with an $L^\infty$ derivative.

A simple closed curve $\gamma$ in the plane is called Lipschitz if there exists a finite number rectangles $\{R_j \}_{j=1}^n$ with sides parallel to the coordinate axes, angles $\{ \theta_j \}_{j=1}^n$ and Lipschitz functions $\varphi_j:\R \to \R$, such that the union $\cup_{j=1}^n \{e^{-i\theta_j} R_j\}$ covers $\gamma$ and the intersection $\{ e^{i\theta_j}(\gamma) \} \cap R_j = \{ x+i\varphi_j(x) : \,\, x \in (a_j,b_j) \}$, for some $a_j < b_j < \infty$.
If $\gamma$ is a simple closed curve in the Riemann sphere passing through $\infty$, say that $\gamma$ is Lipschitz if there is a Möbius transformation $\Phi$ mapping $\gamma$ to a simple closed Lipschitz curve in the plane.

Each simple closed oriented curve $\gamma \subset \hat{\C}$ bounds two simply connected domains: write $\Omega_+$ for the domain lying to the left and $\Omega_-$ the domain to the right.
When $\gamma$ is a planar curve, it is assumed to have counterclockwise orientation unless explicitly stated otherwise; we refer to $\Omega_+$ and $\Omega_-$ as {\em interior} and {\em exterior} domains, respectively.
$\Omega_+$ and $\Omega_-$ are called Lipschitz domains when their boundary $\gamma$ is Lipschitz.
Note that if $\gamma$ is an oriented curve in the Riemann sphere and $\Phi$ is a Möbius transformation with its pole in $\Omega_-$, the image curve $\Phi(\gamma)$ is a planar curve oriented counterclockwise.
When the pole is in $\Omega_+$, the orientation is reversed.

Let $\gamma$ be a simple closed planar curve oriented counterclockwise.
For $\beta > 0$ and $\zeta \in \gamma$, define a set called a {\em non-tangential approach region} to $\zeta$ by
\begin{equation*}
\Gamma(\zeta) = \{z \in \C : |z - \zeta| \le (1+\beta)\,\mathrm{dist}(z,\gamma) \}.
\end{equation*}
Lipschitz curves are well-known to satisfy the {\em uniform interior and exterior cone condition}, meaning there exists $\beta,r>0$ such that for each $\zeta \in \gamma$, one of the two components of $\Gamma(\zeta) \cap D(\zeta,r)$ is contained in $\Omega_+$ and the other contained in $\Omega_-$.
Write the interior and exterior non-tangential approach regions by $\Gamma_\pm(\zeta) = \Gamma(\zeta) \,\cap \, \Omega_\pm$.
An important technical tool for work on Lipschitz domains is a {\em Neças exhaustion}, a method of approximation by $C^\infty$ subdomains with uniformly bounded Lipschitz constants; see \cite{Lan99,Lanzani2000} for details.

Given a function $g:\Omega_\pm \to \C$ and $\zeta \in \gamma$, its non-tangential maximal function $g^*$ and non-tangential limit $\dot{g}$ (when it exists) are defined to be
\begin{equation*}
g^*(\zeta) = \sup_{z \in \Gamma_\pm(\zeta)} |g(z)|, \qquad \qquad \dot{g} (\zeta) = \lim_{\Gamma_\pm(\zeta) \ni z \to \zeta} g(z).
\end{equation*}
Given $f \in L^2(\gamma)$, its Cauchy transform \eqref{E:Cauchy-transform-boundary} arises as the non-tangential limit of the Cauchy integral in \eqref{E:Cauchy-integral-formula}.
A deep and highly non-trivial result in \cite{CMM82} shows that this limit exists a.e. for Lipschitz $\gamma$, and further, defines an $L^2(\gamma)$ function.
We slightly abuse notation by denoting both the Cauchy integral of $f$ and its boundary values by $\bm{C}_\pm f$, but our intended meaning should always be clear from context.

We now {\em define} the Hardy space $\ch_\pm^2(\gamma)$ as the image of $L^2(\gamma)$ under $\bm{C}_\pm$:
\begin{equation}\label{E:def-of-Hardy-space}
\ch_\pm^2(\gamma) = \{\bm{C}_\pm f : f \in L^2(\gamma) \};
\end{equation}
since $\gamma$ is always assumed to be Lipschitz, this definition is equivalent to several other characterizations of the Hardy space used in the literature; see \cite{Lanzani2000}.
We have been intentionally flexible with our definition so that Hardy space functions can be at times thought of as holomorphic functions with $L^2$ boundary values and at other times as the boundary values themselves.
Observe from \eqref{E:Cauchy-integral-formula} that functions in $\ch_-^2(\gamma)$ necessarily vanish at $\infty$.

The results in \cite{CMM82} along with the Plemelj jump formula (see \cite{Mus92}) allow rigorous justification of the following ``intuitive" statements for Lipschitz $\gamma$: if $f \in \ch_\pm^2(\gamma)$, then $\bm{C}_\pm f = f$ (Cauchy's integral formula), while $\bm{C}_\mp f \equiv 0$ (Cauchy's theorem).






\begin{remark}\label{R:density}
Given $\alpha \in (0,1)$, define the space of $\alpha$-Hölder continuous functions on $\gamma$ to be
$$
C^\alpha(\gamma) := \{f : |f(x)-f(y)| < |x-y|^\alpha, \,\, x,y \in \gamma \},
$$
and denote by $A^\alpha(\ol{\Omega_\pm})$ the space of holomorphic functions on $\Omega_\pm$ with $C^\alpha$ boundary values.
If $\gamma$ is Lipschitz and $f \in C^{\alpha}(\gamma)$, then $\bm{C}_\pm f \in A^\alpha(\ol{\Omega_\pm})$; see \cite[Appendix 2]{Mus92}.
The regularity of $\bm{C}_\pm$ in $C^\alpha$ together with its boundedness in $L^2(\gamma)$ imply that $A^\alpha(\ol{\Omega_\pm})$ is a dense subspace of the Hardy space $\ch_\pm^2(\gamma)$.
\hfill $\lozenge$
\end{remark}

\subsection{Dual space characterization}\label{SS:Two-pairings-dual-spaces}

A duality paradigm of Grothendieck \cite{Groth1953a}, Köthe \cite{Koethe1953} and Sebasti\~{a}o e Silva \cite{SebeSilva1950} identifies duals of holomorphic function spaces on simply connected domains with spaces of holomorphic functions on their complements:
Let $\co(\Omega_+)$ denote the space of all holomorphic functions on $\Omega_+$ under the standard Frechét topology.
Under this paradigm, the dual can be identified with $\co_0(\ol{\Omega_-})$, the space of functions holomorphic in a neighborhood of $\Omega_-$ which vanish at $\infty$.
The functionals themselves are represented using bilinear pairings $\<<>>{\cdot}{\cdot}\,$ à la \eqref{E:bilinear-pairing-def} to pair $f \in \co(\Omega_+)$ and $g \in \co_0(\ol{\Omega_-})$, where the path of integration is taken inside $\Omega_+$ and sufficiently close to $\gamma$.

We follow this paradigm and identify the dual space of $\ch_\pm^2(\gamma)$ with $\ch_\mp^2(\gamma)$.

Since $\bm{C}_\pm$ is bounded on $L^2(\gamma)$ whenever $\gamma$ is Lipschitz, a bounded adjoint exists (with respect to the standard inner product), characterized by $\langle \bm{C}_\pm f,g \rangle = \langle f,\bm{C}^*_\pm g \rangle$.
Explicitly,
\begin{equation*}\label{C:Cauchy-adjoint-formula}
\bm{C}^*_{\pm}g(z) =  \frac{g(z)}{2} \pm \frac{1}{2\pi i} \ol{T(z)} \,\mathrm{P.V.} \int_{\gamma} \frac{g(\zeta)}{\ol{\zeta} - \ol{z}}\,  d\sigma(\zeta),
\end{equation*}
where the formula is understood to hold for almost every $z \in \gamma$.



\begin{proposition}
The Cauchy transforms $\bm{C}_+$ and $\bm{C}_-$ can be viewed as ``adjoints" with respect to the bilinear pairing \eqref{E:bilinear-pairing-def}.
Indeed,
\begin{equation*}
\<<>>{\bm{C}_\pm f}{g} = \<<>>{f}{\bm{C}_\mp g} = \<<>>{\bm{C}_\pm f}{\bm{C}_\mp g}.
\end{equation*}
\end{proposition}
\begin{proof}
Since $\bm{C}_\pm$ is a projection operator, it will suffice to prove the first equality.
We claim that if $g \in L^2(\gamma)$ and $T$ is the almost everywhere defined unit tangent vector for $\gamma$, then $\bm{C}^*_\pm(\ol{g T}) = \ol{ \bm{C}_\mp(g) T }$.
Indeed, for a.e. $z \in \gamma$, we have
\begin{align*}
\bm{C}^*_\pm(\ol{g T})(z) 
&= \frac{\ol{g(z)} \ol{T(z)}}{2} \pm \frac{1}{2\pi i} \ol{T(z)} \,\mathrm{P.V.} \int_{\gamma} \frac{\ol{g(\zeta)} \ol{T(\zeta)}}{\ol{\zeta} - \ol{z}}\,  d\sigma(\zeta) \notag \\
&= \left( \frac{\ol{g(z)}}{2} \pm \frac{1}{2\pi i} \,\mathrm{P.V.} \int_{\gamma} \frac{\ol{g(\zeta)} \ol{T(\zeta)} }{\ol{\zeta} - \ol{z}}\,  d\sigma(\zeta) \right) \ol{T(z)} \\
&= \ol{ \left( \frac{g(z)}{2} \mp \frac{1}{2\pi i} \,\mathrm{P.V.} \oint_{\gamma} \frac{g(\zeta)}{\zeta - z}\,  d\zeta \right) T(z) } = \ol{\bm{C}_\mp(g)(z) T(z)}.
\end{align*}
Thus we see that $\,\<<>>{\bm{C}_\pm f}{g} = \langle \bm{C}_\pm f, \ol{g T} \rangle = \langle f, \bm{C}_\pm^* (\ol{g T}) \rangle = \langle f, \ol{ \bm{C}_\mp(g) T } \rangle = \<<>>{f}{\bm{C}_\mp g}$.
\end{proof}

\begin{proposition}\label{P:interior-exterior-dual-identification}
The dual space of $\ch_\pm^2(\gamma)$ can be identified with $\ch_\mp^2(\gamma)$ via functionals $\psi_g: \ch_\pm^2(\gamma) \to \C$, $g \in \ch_\mp^2(\gamma)$, given by $\psi_g(f) = \<<>>{f}{g}$.
Moreover,
\begin{equation}\label{E:range-of-functional-norm}
\norm{\bm{C}_\pm}^{-1} \norm{g} \le \norm{\psi_g}_{\sf op} \le  \norm{g}.
\end{equation}
\end{proposition}
\begin{proof}
Since $\ch_\pm^2(\gamma)$ is a Hilbert space, it is self dual in the ordinary inner product.
Thus, given a bounded linear functional $\phi: \ch_\pm^2(\gamma) \to \C$, there is a unique $h \in \ch_\pm^2(\gamma)$ so that for any $f \in \ch_\pm^2(\gamma)$,
\begin{align*}
\phi(f) = \langle f,h \rangle = \<<>>{f}{\ol{hT}} = \<<>>{ \bm{C}_\pm f}{\ol{hT}} = \<<>>{ f}{\bm{C}_\mp(\ol{hT})}.
\end{align*}
Now set $g = \bm{C}_\mp(\ol{hT}) \in \ch_\mp^2(\gamma)$, so that $\phi = \psi_g = \<<>>{\cdot}{g} \in \ch_\pm^2(\gamma)'$.

Given distinct $g_1,g_2 \in \ch_\mp^2(\gamma)$, we now show the functionals $\psi_{g_1} \neq \psi_{g_2}$.
It will suffice to exhibit an $f \in \ch_\pm^2(\gamma)$ with $\psi_{g_1}(f) \neq \psi_{g_1}(f)$.
Set $f = \bm{C}_\pm \big( \ol{(g_1-g_2)T} \big)$, which is clearly in $\ch_\pm^2(\gamma)$.
Then
\begin{align*}
(\psi_{g_1} - \psi_{g_2})(f) &=\<<>>{f}{g_1-g_2} \\
&= \<<>>{\bm{C}_\pm \big( \ol{(g_1-g_2)T}\big)}{g_1-g_2} \\
&= \<<>>{\ol{(g_1-g_2)T}}{\bm{C}_\mp (g_1-g_2)}
= \<<>>{\ol{(g_1-g_2)T}}{g_1-g_2} 
= \norm{g_1-g_2}^2 > 0.
\end{align*}

We now prove \eqref{E:range-of-functional-norm}.
The right-hand inequality follows from Cauchy-Schwarz.
For the left-hand inequality, note that for $g \in \ch_\mp^2(\gamma)$
\begin{alignat*}{2}
\norm{g}
& = \sup \left\{ \left|\,\<<>>{h}{g}\,\right| : h\in L^2(\gamma),\, \left\| h\right\| = 1 \right \} && \\
&=  \sup \left\{ \left|\,\<<>>{h}{\bm{C}_\mp g}\,\right| : {h\in L^2(\gamma),\, \left\| h\right\| = 1} \right\} 
&&  \\
& = \sup \left\{ \left|\,\<<>>{ \bm{C}_\pm h}{g}\,\right| : {h\in L^2(\gamma),\, \left\| h\right\| = 1} \right\} && \le \sup \left\{ \left|\,\<<>>{f}{g}\,\right| : {f\in\ch_\pm^2(\gamma),\, \left\| f\right\| \le \left\|\bm{C}_\pm\right\| } \right\}\\
& &&= \left\|\bm{C}_\pm\right\| \cdot \sup \left\{ \left|\,\<<>>{f}{g}\,\right| : {f\in\ch_\pm^2(\gamma),\, \left\| f\right\| = 1} \right\}\\
& &&= \left\|\bm{C}_\pm\right\| \cdot \norm{\psi_g}_{\sf op}.
\end{alignat*}
\end{proof}


\begin{theorem}\label{T:norm-pairing-efficieny}
Let $\gamma$ be a simple closed Lipschitz curve in the plane.
The norms of the Cauchy transforms $\bm{C}_\pm : L^2(\gamma) \to \ch^2_\pm(\gamma)$ are given by
\begin{equation}\label{E:Cauchy-norm-eq}
\frac{1}{\norm{\bm{C}_+}} = \inf_{ \substack{g \in \ch_+^2(\gamma) \\ {g \neq 0} }} \Bigg\{ \sup_{\substack{ f \in \ch_-^2(\gamma) \\ f \neq 0 }} \frac{|\<<>>{f}{g}|}{\norm{f}\norm{g}} \Bigg\} 
= \inf_{\substack{g \in \ch_-^2(\gamma) \\ g \neq 0}} \Bigg\{ \sup_{\substack{ f \in \ch_+^2(\gamma) \\ f \neq 0}} \frac{ |\<<>>{f}{g}|}{\norm{f}\norm{g}} \Bigg\}
=\frac{1}{\norm{\bm{C}_-}}.
\end{equation}
\end{theorem}
\begin{proof}
Given a nonzero $g \in \ch_\mp^2(\gamma)$,
the lower bound in \eqref{E:range-of-functional-norm} says
\[
\norm{\bm{C}_\pm}^{-1} \norm{g} \le \norm{\psi_g}_{\sf op} = \sup \left\{|\<<>>{f}{g}|: f \in \ch_{\pm}^2(\gamma), \, \norm{f}=1 \right\}.
\]
As this holds for every such $g$, we obtain
\begin{equation}\label{E:inf-sup-ineq1}
\frac{1}{\norm{\bm{C}_\pm}} \le \inf_{\substack{ g \in \ch_{\mp}^2(\gamma)  \\ g \neq 0}} \Bigg\{ \sup_{\substack{ f \in \ch_{\pm}^2(\gamma)  \\ f \neq 0}} \frac{\left|\<<>>{f}{g} \right|}{\norm{f}\norm{g}} \Bigg\}.
\end{equation}
On the other hand, given (a sufficiently small) $\epsilon > 0$, there exists $h_\epsilon \in L^2(\gamma)$ such that $\norm{\bm{C}_\pm h_\epsilon} = 1$ and $\norm{h_\epsilon} < (\norm{\bm{C}_\pm}-\epsilon)^{-1}$.
Now observe that
\begin{equation*}
\sup_{\substack{f \in \ch_{\mp}^2(\gamma) \\ \norm{f} = 1 }} \left|\<<>>{\bm{C}_\pm h_\epsilon}{f} \right| 
= \sup_{\substack{f \in \ch_{\mp}^2(\gamma) \\ \norm{f} = 1 }} \left|\<<>>{h_\epsilon}{\bm{C}_\mp f} \right|
= \sup_{\substack{f \in \ch_{\mp}^2(\gamma) \\ \norm{f} = 1 }} \left|\<<>>{h_\epsilon}{f} \right|
\le \norm{h_\epsilon} < \frac{1}{\norm{\bm{C}_\pm}-\epsilon}.
\end{equation*}
Taking $g=\bm{C}_\pm h_\epsilon \in \ch_\pm^2(\gamma)$ and letting $\epsilon \to 0$ we obtain
\begin{equation}\label{E:inf-sup-ineq2}
\inf_{\substack{g \in \ch_{\pm}^2(\gamma)  \\ g \neq 0}} \Bigg\{ \sup_{\substack{ f \in \ch_{\mp}^2(\gamma)  \\ f \neq 0}} \frac{\left|\<<>>{g}{f} \right|}{\norm{g}\norm{f}} \Bigg\}
\le \frac{1}{\norm{\bm{C}_\pm}}.
\end{equation}
Now combine all four individual inequalities in \eqref{E:inf-sup-ineq1} and \eqref{E:inf-sup-ineq2} to obtain 
\begin{equation*}
\frac{1}{\norm{\bm{C}_+}} 
\le \inf_{\substack{ g \in \ch_{-}^2(\gamma)  \\ g \neq 0}} \Bigg\{ \sup_{\substack{ f \in \ch_{+}^2(\gamma)  \\ f \neq 0}} \frac{\left|\<<>>{f}{g} \right|}{\norm{f}\norm{g}} \Bigg\} 
\le \frac{1}{\norm{\bm{C}_-}} 
\le \inf_{\substack{ g \in \ch_{+}^2(\gamma)  \\ g \neq 0}} \Bigg\{ \sup_{\substack{ f \in \ch_{-}^2(\gamma)  \\ f \neq 0}} \frac{\left|\<<>>{f}{g} \right|}{\norm{f}\norm{g}} \Bigg\} 
\le \frac{1}{\norm{\bm{C}_+}},
\end{equation*}
forcing equality to hold at every step.
\end{proof}

\subsection{The Szeg\H{o} kernel}\label{SS:Szego-kernel}
Several elementary properties are collected here for later use.

\begin{proposition}[\cite{Bell16}, Chapter 7]
The Szeg\H{o} kernel on the unit disc $\D$ is
\begin{equation}\label{E:Szego-on-disc}
S_\D(z,\zeta) = \frac{1}{2\pi(1-z\ol{\zeta})}, \qquad z \in \D,\quad  \zeta\in \ol{\D}.
\end{equation}
\end{proposition}

The Szeg\H{o} kernel admits a biholomorphic transformation law; see \cite[Theorem 12.2]{Bell16} in the $C^\infty$ setting, and \cite[Lemma 5.3]{Lan99} for the Lipschitz setting:

\begin{proposition}\label{P:SzegoTransLaw}
Let $\Phi: \Omega_1 \to \Omega_2$ be a biholomorphism of simply connected domains in the Riemann sphere with Lipschitz boundaries.  
The Szeg\H{o} kernels are related by formula
\begin{equation}\label{E:SzegoTransLaw}
S_1(z,\zeta) = \sqrt{\Phi'(z)} \cdot S_2(\Phi(z),\Phi(\zeta)) \cdot  \overline{\sqrt{\Phi'(\zeta)}}.
\end{equation}
\end{proposition}

The Szeg\H{o} kernel admits a well-known extremal property; see \cite[Sections 1.4, 1.5]{Krantz_scv_book}:
\begin{proposition}\label{P:Szego-diagonal-extremal-char}
Given a simple closed Lipschitz curve $\gamma$ in the Riemann sphere and a point $z \in \Omega_\pm$, the Szeg\H{o} kernel satisfies
\begin{equation}\label{E:Szego-diagonal-extremal-char}
S_\pm(z,z) = \sup\{|f(z)|: f \in \ch^2_\pm(\gamma), \,\, \norm{f}_{L^2(\gamma)} = 1 \}.
\end{equation}
\end{proposition}

\begin{remark}\label{R:positivity-of-S(z,z)}
In the setting of Proposition \ref{P:Szego-diagonal-extremal-char}, the Riemann mapping theorem together with formulas \eqref{E:Szego-on-disc} and \eqref{E:SzegoTransLaw} show that $S_\pm(z,z) > 0$ for any  $z \in \Omega_\pm\setminus\{\infty\}$. 
On the other hand, the condition that functions in the Hardy space must vanish at infinity shows that if $\infty \in \Omega_\pm$, then $S_\pm(\infty,\infty) = 0$.
\hfill $\lozenge$
\end{remark}

The following monotonicity property is known, but a short proof is included since the authors had difficulty locating a reference.

\begin{proposition}\label{P:Szego-containment-monotonicity}
Let $\Omega_1 \subsetneq \Omega_2 \subsetneq \hat\C$ be simply connected domains with Lipschitz boundaries properly contained in the Riemann sphere, and let $z \in \Omega_1\cut\{\infty\}$.
Letting $S_1, S_2$ denote the respective Szeg\H{o} kernels, we have
\begin{equation}
0< S_2(z,z) < S_1(z,z).
\end{equation}
\end{proposition}
\begin{proof}
Let $\Phi_j:\Omega_j \to \D$ denote the Riemann map, $j=1,2$, with $\Phi_j(z) = 0$ and $\Phi_j'(z) > 0$.
Using the transformation law in \eqref{E:SzegoTransLaw} and the kernel formula for $\D$ in \eqref{E:Szego-on-disc}, we see
\begin{equation*}
2\pi S_j(z,z) = \Phi_j'(z).
\end{equation*}
By the proof of the Riemann mapping theorem (see, e.g., \cite[Chapter 6]{AhlforsBook}), of all maps from $\Omega_1$ into the disc $\D$ satisfying $\Phi(z) = 0$ and $\Phi'(z)$ positive, the Riemann map $\Phi_1$ is uniquely determined by the property that $\Phi'(z)$ is maximal.
Since the restriction of $\Phi_2$ to $\Omega_1$ is also a map with these properties, we conclude that $\Phi_2'(z) < \Phi_1'(z)$.
\end{proof}

\subsection{A lower estimate on the norm of the Cauchy transform}

\begin{theorem}\label{T:Lambda-function-lower-bound}
Let $\gamma$ be a simple closed Lipschitz curve in the plane and $z \in \C$.
Then 
\begin{equation*}
\Lambda(\gamma,z) \le \norm{\bm{C}_\pm}.
\end{equation*}
\end{theorem}
\begin{proof}
For $z \in \Omega_\pm \backslash \{\infty\}$, define $h_z \in \ch^2_{\mp}(\gamma)$ by $h_z(\zeta) = (2\pi i(\zeta-z))^{-1}$.
By
Cauchy's integral formula we have
\begin{equation*}
\<<>>{f}{h_z} = \frac{1}{2\pi i}\oint_\gamma\frac{f(\zeta)}{\zeta-z}\,d\zeta = f(z), \qquad f \in \ch^2_\pm(\gamma).
\end{equation*}
Now apply the Cauchy norm characterization in \eqref{E:Cauchy-norm-eq} with $g=h_z$ to obtain
\begin{equation*}
\frac{1}{\norm{\bm{C}_\pm}}  
\le \sup_{f \in \ch^2_\pm(\gamma)} \frac{| \<<>>{f}{h_z} |}{\norm{f} \norm{h_z} } 
= \frac{1}{\norm{C(z,\cdot)}}  \sup_{f \in \ch^2_\pm(\gamma)} \frac{|f(z)|}{\norm{f}} 
= \frac{\sqrt{S_\pm(z,z)}}{\norm{C(z,\cdot)}} =\frac{1}{\Lambda(\gamma,z)},
\end{equation*}
where we used the extremal property of \eqref{E:Szego-diagonal-extremal-char}.
This estimate holds for all $z \in \C\setminus\gamma$. 
Since $\Lambda(\gamma,\cdot) \equiv 1$ for $z \in \gamma$, the result follows for these $z$ from the fact that $\bm{C}_\pm$ is a projection onto $\ch_\pm^2(\gamma)$ and thus $\norm{\bm{C}_\pm} \ge 1$.
\end{proof}


\section{Invariance and rigidity properties}\label{S:Properties-of-Lambda}

\subsection{Analytic capacity and behavior at infinity}\label{SS:Analytic-Capacity}

Let $\gamma$ be a simple closed Lipschitz curve in the plane oriented counterclockwise.
If $g$ is holomorphic on the exterior domain $\Omega_-$, it admits a Laurent expansion in a neighborhood of $\infty$:
\begin{equation*}
g(z) = a_0 + a_1 z^{-1} + a_2 z^{-2} + \cdots
\end{equation*}
The coefficient $a_1$ is important to what comes below; it can be obtained by calculating the derivative of $g$ at infinity with respect to the local coordinate $\tfrac{1}{z}$.
Define
\begin{equation}\label{E:derivative-at-infinty}
D(g,\infty) := \lim_{z \to \infty} z(g(z) - g(\infty)) = a_1.
\end{equation}
(In the literature, $D(g,\infty)$ is often denoted by $g'(\infty)$, but the authors find this notation misleading since $\lim_{z\to \infty} g'(z) \neq D(g,\infty)$ unless $a_1 = 0$.)

Let $A^\infty(\Omega_-)$ be the space of bounded holomorphic functions on $\Omega_-$, with norm given by $\norm{g}_\infty := \sup \{|g(z)|: z \in \Omega_- \}$.
Define the {\em analytic capacity} of the curve $\gamma$ to be
\begin{equation}\label{E:def-of-analytic-capacity}
\kappa(\gamma) := \sup\{|D(g,\infty)| : g \in A^\infty(\Omega_-),\quad g(\infty) = 0,\quad \norm{g}_\infty \le 1 \}.
\end{equation}
This notion helps formulate generalizations of Riemann's removable singularity theorem by measuring how large bounded holomorphic functions on $\Omega_-$ can become;
see \cite{Garnett1972,Ransford1995}.

\begin{theorem}\label{T:Lambda-at-infinity}
Let $\gamma$ be a simple closed Lipschitz curve in the plane. 
Then
\begin{equation}\label{E:Bolt-lower-bound}
\lim_{z \to \infty} \Lambda_-(\gamma,z) = \sqrt{\frac{\sigma(\gamma)}{2\pi \kappa(\gamma)}},
\end{equation}
where $\sigma(\gamma)$ and $\kappa(\gamma)$ denote the arc length and analytic capacity of $\gamma$, respectively.
Thus $\Lambda(\gamma,\cdot)$ is continuous at $\infty$ (by definition).
\end{theorem}
\begin{proof}
Set $E := \{z\in \C : z^{-1} \in \Omega_- \}$, which is a bounded domain containing the origin.

Define a holomorphic and univalent function $G:E \to \D$ with the following properties: 
$(i)$ $\norm{G}_\infty \le 1$;
$(ii)$ $G(0) = 0$;
$(iii)$ $G'(0)$ is positive and maximal, i.e., given another map $H: E \to \D$ satisfying $(i)$ and $(ii)$ with $H'(0)$ positive, then necessarily $G'(0) > H'(0)$.
Such a $G$ always exists and is the Riemann map (see \cite[Section 6.1]{AhlforsBook}) from $E$ to $\D$ satisfying $G(0) = 0$ with $G'(0)>0$.
Now write $G$ as a Taylor expansion about 0:
\begin{equation*}
G(z) =  a_1 z + a_2 z^2 + \cdots
\end{equation*}

Now define a biholomorphic map $g:\Omega_- \to \D$ by $g(z) = G(\frac{1}{z})$.
Clearly $(i')$ $\norm{g}_\infty \le 1$; 
and $(ii')$ $g(\infty) = 0$.
We claim the positive number $D(g,\infty)$ defined by \eqref{E:derivative-at-infinty} is maximal out of all functions in $A^\infty(\Omega_-)$ satisfying $(i')$ and $(ii')$.
If $D(g,\infty)$ weren't maximal, there would exist an $h \in A^\infty(\Omega_-)$
with $D(h,\infty) > D(g,\infty) = a_1$.
But then the function $H(z) := h(\frac{1}{z})$ would satisfy $(i)$ and $(ii)$ from the previous paragraph, and $H'(0) > a_1 = G'(0)$, contradicting the maximality of $G'(0)$.
Therefore, $\kappa(\gamma) = D(g,\infty) = \lim_{z\to\infty} z g(z) = a_1 = G'(0)$.

Now use Proposition \ref{P:SzegoTransLaw} and \eqref{E:Szego-on-disc} to write the Szeg\H{o} kernel of $\Omega_-$:
\begin{equation*}
S_-(z,z) = |g'(z)| S_\D(g(z),g(z)) = \frac{1}{2\pi}\cdot \frac{|g'(z)|}{1-|g(z)|^2}.
\end{equation*} 
Thus,
\begin{align*}
\Lambda_-(\gamma,z)^2
= \frac{\norm{C(z,\cdot)}_{L^2(\gamma)}^2}{S_-(z,z)} 
&= \left(|z|^2 \int_\gamma |C(z,\zeta)|^2\,d\sigma(\zeta) \right) \left( \frac{1}{2\pi}\cdot \frac{|z|^2|g'(z)|}{1-|g(z)|^2} \right)^{-1}, \notag
\end{align*}
where the term $|z|^2$ has been inserted in both the numerator and denominator.
Now,
\begin{equation}\label{E:intermediate-factors-Lambda_-_1}
\lim_{z \to \infty}  |z|^2 \int_\gamma |C(z,\zeta)|^2\,d\sigma(\zeta) 
= \lim_{z \to \infty}  \frac{1}{4\pi^2} \int_\gamma \frac{d\sigma(\zeta)}{|\frac{\zeta}{z}-1|^2} 
= \frac{\sigma(\gamma)}{4\pi^2}.
\end{equation}
On the other hand, 
\begin{equation}\label{E:intermediate-factors-Lambda_-_2}
\lim_{z \to \infty} \frac{1}{2\pi}\cdot \frac{|z|^2|g'(z)|}{1-|g(z)|^2} 
= \frac{1}{2\pi} \lim_{z \to \infty} \frac{|z^2 g'(z)|}{1-|g(z)|^2} 
= \frac{a_1}{2\pi}
= \frac{\kappa(\gamma)}{2\pi}.
\end{equation}
Dividing \eqref{E:intermediate-factors-Lambda_-_1} by \eqref{E:intermediate-factors-Lambda_-_2} gives the result.
\end{proof}

\begin{remark}\label{R:Bolt-lower-bound}
In \cite[Theorem~1]{Bol07} Bolt carries out a similar computation, obtaining a lower bound of the norm of the Kerzman-Stein operator.
\hfill $\lozenge$
\end{remark}

\subsection{Möbius Invariance}\label{SS:Möbius-invariance}

Recall that the holomorphic automorphisms of the Riemann sphere are precisely the Möbius transformations
\begin{equation}\label{E:Möbius-form}
\Phi(z) = \frac{az+b}{cz+d},
\end{equation}
where $a,b,c,d \in \C$ with $ad-bc \neq 0$.
The Cauchy kernel and transform admit transformation laws under these maps.
See \cite[Theorem 3]{Bol05} for an analogous result in $\C^n$ (or more accurately $\C\mathbb{P}^n$) on the projective invariance of the Leray kernel.

\begin{theorem}\label{T:Cauchy-transform-Möbius-invariance}
Let $\gamma_1$ be a simple closed Lipschitz curve in the complex plane oriented counterclockwise and let $\Phi$ be a Möbius transformation 
whose pole lies off of $\gamma_1$.
Define the curve $\gamma_2 = \Phi(\gamma_1)$ with orientation induced from the orientation of $\gamma_1$ by $\Phi$; thus $\gamma_2$ will be oriented counterclockwise if and only if the pole of $\Phi$ lies in $\Omega_-$.
Let $\bm{C}_\pm^1$ and $\bm{C}_\pm^2$ denote the Cauchy transforms of $\gamma_1$ and $\gamma_2$, respectively.
Then
\begin{equation}\label{E:Cauchy-transform-Möbius-invariance}
\bm{C}_\pm^1 \left(\sqrt{\Phi'}\cdot(f\circ\Phi) \right) = \sqrt{\Phi'}\cdot\big( (\bm{C}_\pm^2 f)\circ\Phi \big), \qquad f \in L^2(\gamma_2).
\end{equation}
\end{theorem}

\begin{proof}
Differentiate \eqref{E:Möbius-form} and observe that $\Phi'$ is the square of a meromorphic function defined on the Riemann sphere.
Now choose a value of $\sqrt{ad-bc}$ and then set
\begin{equation}\label{E:Sqrt-Phi'}
\sqrt{\Phi'(\zeta)} = \frac{\sqrt{ad-bc}}{c\zeta+d}.
\end{equation}

Observe that the map $f \mapsto \sqrt{\Phi'}\cdot (f\circ\Phi)$ is a linear isomorphism from $L^2(\gamma_2)$ to $L^2(\gamma_1)$.
Now let $f \in L^2(\gamma_2)$, $\zeta \in \gamma_1$ and $\xi = \Phi(\zeta) \in \gamma_2$.
If $z \in \Omega^1_\pm$, then the image point $\Phi(z) \in \Phi(\Omega^1_\pm) = \Omega_\pm^2$ and
\begin{align}
(\bm{C}_\pm^2 f) \circ \Phi (z) = \frac{1}{2\pi i} \oint_{\gamma_2} \frac{f(\xi)}{\xi - \Phi(z)} \,d\xi
&= \frac{1}{2\pi i} \oint_{\gamma_1} \frac{f(\Phi(\zeta))}{\Phi(\zeta) - \Phi(z)} \cdot \Phi'(\zeta) \,d\zeta \label{E:CauchyMöbiusInv-0} \\
&= \frac{1}{2\pi i} \oint_{\gamma_1} \frac{f(\Phi(\zeta))}{\frac{a\zeta+b}{c\zeta+d} - \frac{az+b}{cz+d}} \cdot \frac{ad-bc}{(c\zeta+d)^2} \,d\zeta. \label{E:CauchyMöbiusInv-1}
\end{align}
Rearranging,
\begin{align}
\eqref{E:CauchyMöbiusInv-1} &= \frac{1}{2\pi i} \oint_{\gamma_1} \frac{(cz+d)(c\zeta+d)(ad-bc)}{(ad-bc)(\zeta-z)(c\zeta+d)^2}\, f(\Phi(\zeta))\, d\zeta \notag \\
&= \frac{1}{2\pi i} \frac{cz+d}{\sqrt{ad-bc}} \oint_{\gamma_1} \frac{\sqrt{ad-bc}}{(c\zeta+d)}\, \frac{f(\Phi(\zeta))}{(\zeta-z)} \, d\zeta \notag \\
&= \frac{1}{2\pi i \sqrt{\Phi'(z)}} \oint_{\gamma_1} \frac{ \sqrt{\Phi'(\zeta)} f(\Phi(\zeta))}{\zeta-z} \, d\zeta = \frac{1}{\sqrt{\Phi'(z)}}\, \bm{C}_\pm^1\left(\sqrt{\Phi'}(f\circ\Phi)\right)(z),
\label{E:CauchyMöbiusInv-2}
\end{align}
giving the result when $z \in \Omega_\pm$.

The argument when $z \in \gamma_1$ follows the same lines except that the integrals must be interpreted in the principle value sense.
For $\epsilon>0$ let $\gamma_{1,\epsilon}:= \gamma_1 \setminus D(z,\epsilon)$, i.e., the original curve with all points within $\epsilon$ of $z$ removed.
Now start from the integral in \eqref{E:CauchyMöbiusInv-2} evaluated over the truncated curve $\gamma_{1,\epsilon}$, and work backwards to \eqref{E:CauchyMöbiusInv-0}:
\begin{align}
\frac{1}{2\pi i} \, \mathrm{P.V.}\oint_{\gamma_1} \frac{\sqrt{\Phi'(\zeta)}f(\Phi(\zeta))}{\zeta-z}\,d\zeta 
&= \lim_{\epsilon \to 0}\frac{1}{2\pi i} \oint_{\gamma_{1,\epsilon}} \frac{\sqrt{\Phi'(\zeta)}f(\Phi(\zeta))}{\zeta-z}\,d\zeta \label{E:principle-value-start} \\ 
&= \lim_{\epsilon \to 0}\frac{\sqrt{\Phi'(z)}}{2\pi i} \oint_{\Phi(\gamma_{1,\epsilon})} \frac{f(\xi)}{\xi-\Phi(z)}\,d\zeta. \label{E:principle-value-candidate}
\end{align}

We claim that the integral in \eqref{E:principle-value-candidate} is also a principle value integral in the ordinary sense.
Indeed, the two endpoints of the truncated curve $\Phi(\gamma_{1,\epsilon})$ approach the point $\Phi(z)$ at the same rate as $\epsilon \to 0$ as a consequence of the fact that the image of the disc $D(z,\epsilon)$ under $\Phi$ tends asymptotically to the disc $D(\Phi(z),|\Phi'(z)|\epsilon)$ as $\epsilon \to 0$.
This means that by setting $\gamma_{2,\delta}:= \gamma_2 \setminus D(\Phi(z),\delta)$ with $\delta := |\Phi'(z)|\epsilon$,
\begin{align}
\eqref{E:principle-value-candidate} = \lim_{\epsilon \to 0} \frac{\sqrt{\Phi'(z)}}{2\pi i} \oint_{\Phi(\gamma_{1,\epsilon})} \frac{f(\xi)}{\xi-\Phi(z)}\,d\zeta
&= \lim_{\delta \to 0} \frac{\sqrt{\Phi'(z)}}{2\pi i} \oint_{\gamma_{2,\delta}} \frac{f(\xi)}{\xi-\Phi(z)}\,d\zeta \notag \\ 
&= \frac{\sqrt{\Phi'(z)}}{2\pi i} \, \mathrm{P.V.} \oint_{\gamma_2} \frac{f(\xi)}{\xi-\Phi(z)}\,d\zeta. \label{E:principle-value-confirmed}
\end{align}
Thus, the string of equalities from \eqref{E:principle-value-start} to \eqref{E:principle-value-confirmed} shows
\begin{align*}
\bm{C}_\pm^1 \left(\sqrt{\Phi'}\cdot(f\circ\Phi) \right)(z) &= \frac{\sqrt{\Phi'(z)}\cdot f(\Phi(z))}{2} \pm \frac{1}{2\pi i} \, \mathrm{P.V.}\oint_{\gamma_1} \frac{\sqrt{\Phi'(\zeta)}f(\Phi(\zeta))}{\zeta-z}\,d\zeta \\
&= \frac{\sqrt{\Phi'(z)}\cdot f(\Phi(z))}{2} \pm \frac{\sqrt{\Phi'(z)}}{2\pi i} \, \mathrm{P.V.} \oint_{\gamma_2} \frac{f(\xi)}{\xi-\Phi(z)}\,d\zeta \\
&= \sqrt{\Phi'(z)} \cdot \left((\bm{C}^2_\pm f) \circ \Phi\right)(z).
\end{align*}
\end{proof}


\begin{theorem}\label{T:Cauchy-Mobius-Inv}
Suppose $\gamma_1$ is simple closed Lipschitz curve in the plane oriented counterclockwise and that $\Phi$ is a Möbius transformation whose pole lies off of $\gamma_1$.
Define the curve $\gamma_2 = \Phi(\gamma_1)$ 
(oriented as in Theorem \ref{T:Cauchy-transform-Möbius-invariance}) and let  $C_\pm^1(z,\zeta)$ and $C_\pm^2(z,\zeta)$ denote the Cauchy kernels of $\gamma_1$ and $\gamma_2$, respectively.
Then
\begin{equation}\label{E:Cauchy-Mobius-Inv}
C_\pm^1(z,\zeta) = \sqrt{\Phi'(z)} \cdot C_\pm^2(\Phi(z),\Phi(\zeta)) \cdot  \overline{\sqrt{\Phi'(\zeta)}}.
\end{equation}
\end{theorem}
\begin{proof}
Since both curves are Lipschitz, tangent vectors exist almost everywhere.
If $\zeta(t)$ parameterizes $\gamma_1$, then $\Phi(\zeta(t))$ parameterizes $\gamma_2$.
The unit tangent to $\gamma_1$ can be written as $T_1(\zeta(t)) = \zeta'(t)/|\zeta'(t)|$, and so the unit tangent to $\gamma_2$ can be written
\begin{equation*}
T_2(\Phi(\zeta(t))) = \frac{\Phi'(\zeta(t))\cdot \zeta'(t)}{|\Phi'(\zeta(t))\cdot \zeta'(t)|} = \frac{\Phi'(\zeta(t))}{|\Phi'(\zeta(t))|} T_1(\zeta(t)).
\end{equation*}
Going forward, we omit reference to the parameter $t$.

Assume $\Phi$ takes the form \eqref{E:Möbius-form}, with $ad-bc \neq 0$, and choose a value of $\sqrt{ad-bc}$ as in \eqref{E:Sqrt-Phi'} to obtain a meromorphic square root of $\Phi$ defined on all of the Riemann sphere.
From the definition of the Cauchy kernel in \eqref{E:Cauchy-kernel-def}, we have
\begin{align}
\sqrt{\Phi'(z)} \cdot C_\pm^2(\Phi(z),\Phi(\zeta)) \cdot \ol{\sqrt{\Phi'(\zeta)}} 
&= \pm \sqrt{\Phi'(z)} \cdot \frac{T_2(\Phi(\zeta))}{\Phi(\zeta) - \Phi(z)} \cdot \ol{\sqrt{\Phi'(\zeta)}} \notag \\
&= \pm \frac{\sqrt{\Phi'(z)} \sqrt{\Phi'(\zeta)}}{\Phi(\zeta) - \Phi(z)}\,  T_1(\zeta). \label{E:Cauchy-kernel-intermediate-step}
\end{align}
A simple computation now shows
\begin{equation}\label{E:Cauchy-kernel-intermediate-step-2}
\frac{\sqrt{\Phi'(z)}\sqrt{\Phi'(\zeta)}}{\Phi(\zeta)-\Phi(z)} 
= \frac{(ad-bc)}{(c\zeta+d)(cz+d)}  \left(\frac{a\zeta+b}{c\zeta+d} - \frac{az+b}{cz+d} \right)^{-1} 
= \frac{1}{\zeta - z}.
\end{equation}
\end{proof}

We now prove that $\Lambda(\gamma,z)$ is Möbius invariant.
This in particular shows that $\Lambda(\gamma,z)$ is well-defined when $\gamma$ is an unbounded Lipschitz curve (recall the discussion of extending $\Lambda$ to unbounded curves following Theorem \ref{T:intro-3-props-of-Lambda}).

\begin{theorem}\label{T:Lambda-Proj-Inv}
Suppose $\gamma$ is a simple closed Lipschitz curve in the plane and $\Phi$ is a Möbius transformation whose pole lies off of $\gamma$. 
Then for $z$ in the Riemann sphere,
\begin{equation*}
\Lambda(\gamma,z) = \Lambda(\Phi(\gamma),\Phi(z)).
\end{equation*}
\end{theorem}
\begin{proof}
Under the assumption on $\Phi$, observe that the image curve $\Phi(\gamma)$ is also a simple closed Lipschitz curve in the plane.
Now write $\gamma_1 := \gamma$ and $\gamma_2 := \Phi(\gamma_1)$.

If $z \in \gamma_1$, then $\Phi(z) \in \gamma_2$, so by definition $\Lambda(\gamma_1,z) = 1 = \Lambda(\gamma_2,\Phi(z))$.

Let $\Omega^j_\pm$ be the domains bounded by $\gamma_j$ and suppose $z \in \Omega^1_\pm$.
By Theorem \ref{T:Cauchy-Mobius-Inv},
\begin{align}
\norm{C_\pm^1(z,\cdot)}_{L^2(\gamma_1)}^2 
&= |\Phi'(z)| \int_{\gamma_1} |C_\pm^2(\Phi(z),\Phi(\zeta))|^2 |\Phi'(\zeta)| \,d\sigma(\zeta) \notag \\
&= |\Phi'(z)| \int_{\gamma_2} |C_\pm^2(\Phi(z),\xi)|^2\,d\sigma(\xi)
= |\Phi'(z)|\cdot \norm{C_\pm^2({\Phi(z)},\cdot)}_{L^2(\gamma_2)}^2. \label{E:Cauchy-int-trans}
\end{align}

Now denote the Szeg\H{o} kernel of $\ch_\pm^2(\gamma_j)$ by $S_\pm^j$.
Since $\Phi$ is a biholomorphism from $\Omega^1_\pm$ to $\Omega^2_\pm$, Proposition~\ref{P:SzegoTransLaw} shows $S^1_\pm(z,z) = |\Phi'(z)| \cdot S^2_\pm(\Phi(z),\Phi(z))$.
This with \eqref{E:Cauchy-int-trans} shows
\begin{equation*}
\Lambda_\pm(\gamma_1,z) 
= \frac{\norm{C_1(z,\cdot)}_{L^2(\gamma_1)}^2}{S^1_\pm(z,z)} 
= \frac{|\Phi'(z)| \cdot \norm{C_2(\Phi(z),\cdot)}_{L^2(\gamma_2)}^2}{|\Phi'(z)| \cdot S^2_\pm(\Phi(z),\Phi(z))} = \Lambda_\pm(\gamma_2,\Phi(z)).
\end{equation*}

The ratio above needs slightly more care in two cases: $(i)$ when $z=\infty$, meaning that $\Phi'(z) = 0$, and $(ii)$ when $\Phi(z) = \infty$, implying that $\Phi'(z) = \infty$.
In either case, the indeterminate ratio is only problematic at this specific $z$; in a punctured neighborhood of $z$, the ratio is valid.
The result now follows by working nearby and then taking limits, in which case we invoke Theorem \ref{T:Lambda-at-infinity} on the continuity of $\Lambda(\gamma,z)$ as $z \to \infty$.
\end{proof}




\subsection{Circles and rigidity}\label{SS:circles-as-extremal-curves}

Circles are shown to be the unique class of extremal curves which globally minimize $\Lambda$.
This leads to interesting rigidity results, including a strengthened version of a famous observation made by Kerzman and Stein; see Corollary \ref{C:vanishing-of-the-Kerzman-Stein-kernel}. 

\begin{proposition}\label{P:Lambda-on-circles=1}
If $\gamma$ is a circle (or a line, including the point at $\infty$), then $\Lambda(\gamma,\cdot) \equiv 1$.
\end{proposition}
\begin{proof}
First let $\gamma = b\D$ be the unit circle. Then \eqref{E:Lambda-def} and \eqref{E:Szego-on-disc} show
\begin{align*}
\Lambda(b\D,0)^2 = \frac{1}{4\pi^2 S_\D(0,0)} \int_\gamma \frac{d\sigma(\zeta)}{|\zeta|^2} = \frac{1}{2\pi} \cdot 2\pi = 1.
\end{align*}
Given $z \in \C \backslash b\D$, consider the Möbius transformation $\varphi_z(w) = \frac{z-w}{1-\ol{z}w}$.
If $|z|<1$ then $\varphi_z$ is an automorphism of $\D$ and if $|z|>1$ then $\varphi_z$ is a biholomorphic map from $\D$ onto $\hat{\C} \backslash \ol{\D}$.
In either case $\varphi_z(0) = z$.
Theorem \ref{T:Lambda-Proj-Inv} now shows $1 = \Lambda(b\D,0) = \Lambda(\varphi_z(b\D),\varphi_z(0)) = \Lambda(b\D,z)$.
For $z=\infty$, use the map $\phi_\infty(w) = w^{-1}$ and repeat the argument above to see $\Lambda(b\D,\infty) = 1$.

Now let $\gamma \subset \hat{\C}$ be any circle and $z \in \Omega_\pm$.
Then there is a Möbius transformation taking $\gamma$ to $b\D$; see \cite[Section 3.3]{AhlforsBook}.
Theorem \ref{T:Lambda-Proj-Inv} implies $\Lambda(\gamma,z) = \Lambda(b\D,\Phi(z)) = 1$.
Since $z$ was chosen arbitrarily, we conclude $\Lambda(\gamma,\cdot) \equiv 1$.
\end{proof}

\begin{theorem}\label{T:Lambda>=1}
Let $\gamma$ be a simple closed Lipschitz curve in the Riemann sphere.
\begin{enumerate}
\item $\Lambda(\gamma,z) \ge 1$, for all $z \in \Omega_\pm$.
\item If there is a single $z \in \Omega_\pm$ such that $\Lambda(\gamma,z) = 1$, then $\Lambda(\gamma,\cdot) \equiv 1$ and $\gamma$ is a circle (or a line, including the point at $\infty$).
\end{enumerate}
\end{theorem}

\begin{proof}
First suppose that $\gamma$ is a planar curve enclosing the bounded domain $\Omega_+$.
We may assume that $z \in \Omega_+$, thanks to the Möbius invariance of $\Lambda$ established in Theorem \ref{T:Lambda-Proj-Inv}.

Consider the Riemann map $g:\D \to \Omega_+$ with $g(0) = z$ and $g'(0) >0$. 
Proposition \ref{P:SzegoTransLaw} and \eqref{E:Szego-on-disc} show 
\begin{equation}\label{E:Szego-kernel-comp1}
\frac{1}{2\pi} = S_\D(0,0) = \sqrt{g'(0)}\, S_+(g(0),g(0)) \ol{\sqrt{g'(0)}} = g'(0) \,S_+(z,z).
\end{equation}

Now let $\Phi_z(w) = \frac{1}{w-z}$ and define the (unbounded) domain $E = \{ \Phi_z(w): w\in \Omega_+ \}$,
along with the map $h = \Phi_z \circ g : \D \to E$.
The Riesz-Privalov theorem \cite[Section 6.3]{Pommerenke1992} says that $g'$ is contained in the Hardy space $H^1(b\D)$, so in particular, it is integrable on the circle.
The norm of the Cauchy kernel is thus
\begin{equation*}
\norm{C(z,\cdot)}_{L^2(\gamma)}^2 
= \frac{1}{4\pi^2}\int_\gamma \frac{d\sigma(\zeta)}{|\zeta-z|^2} 
= \frac{1}{4 \pi^2} \int_{b\D} \frac{|g'(\zeta)|}{|g(\zeta)-z|^2}\,d\sigma(\zeta) 
= \frac{1}{4 \pi^2} \int_{b\D} |h'(\zeta)|\,d\sigma(\zeta). 
\end{equation*}

Now combine this with \eqref{E:Szego-kernel-comp1}:
\begin{equation}\label{E:Lambda-in-terms-of-h}
\Lambda(\gamma,z)^2 = \frac{\norm{C(z,\cdot)}^2_{L^2(\gamma)}}{S_+(z,z)} = \frac{g'(0)}{2\pi} \int_{b\D} |h'(\zeta)|\,d\sigma(\zeta). 
\end{equation}

The conditions on $g$ show that
\begin{align*}
\frac{1}{h(\zeta)} = g(\zeta) - z &= g'(0)\zeta + \frac{g''(0)}{2}\zeta^2 
+ \cdots = g'(0)\zeta \cdot F_1(\zeta),
\end{align*}
where $F_1$ is a non-vanishing holomorphic function on $\D$ with $F_1(0) = 1$.
Thus,
\begin{equation*}
h(\zeta) = \frac{1}{g'(0)\zeta\cdot F_1(\zeta)} = \frac{1}{g'(0)\zeta} + F_2(\zeta),
\end{equation*}
where $F_2$ is holomorphic on the unit disc.
Consequently,
\begin{equation}\label{E:zeta-h'(zeta)}
\zeta h'(\zeta) = -\frac{1}{g'(0)\zeta} + \zeta F'_2(\zeta).
\end{equation}

The residue theorem now shows
\begin{align*}
1 = -\re\left( \frac{g'(0)}{2\pi i} \oint_{b\D} \zeta h'(\zeta)\,d\zeta \right) 
&= -\re\left(\frac{g'(0)}{2\pi} \int_0^{2\pi} e^{2 i \theta} h'(e^{i\theta})\,d\theta \right) \\
&\le \frac{g'(0)}{2\pi} \int_0^{2\pi} |h'(e^{i\theta})|  \,d\theta \label{E:int-est-sharp-ineq} 
= \frac{g'(0)}{2\pi} \int_{b\D} |h'(\zeta)|\,d\sigma(\zeta) 
= \Lambda(\gamma,z)^2. \notag
\end{align*}

From these computations, $\Lambda(\gamma,z) = 1$ if and only if $e^{2 i \theta} h'(e^{i\theta}) \le 0$ for all $\theta \in [0,2\pi]$, which happens if and only if $\phi(\zeta) := \zeta^2 h'(\zeta) \le 0$ for all $\zeta \in b\D$.
Equation \eqref{E:zeta-h'(zeta)} shows $\phi$ extends holomorphically to the origin, with $\phi(0) = - g'(0)^{-1}$.
Since $\phi$ is real-valued on $b\D$ the Schwarz Reflection Principle applies, yielding a bounded holomorphic extension of $\phi$ to the entire complex plane, which means that $\phi$ is necessarily constant ($\phi  \equiv - g'(0)^{-1}$).

Thus $h'(\zeta) = - g'(0)^{-1}\zeta^{-2}$, meaning that $h(\zeta) = g'(0)^{-1}\zeta^{-1} + C$ for some constant $C$. 
This shows that $g(\zeta) = z + h(\zeta)^{-1}$ is a Möbius transformation and therefore $\gamma = g(b\D)$ is a circle. 
Proposition \ref{P:Lambda-on-circles=1} now shows that $\Lambda(\gamma,\cdot) \equiv 1$.

Now if $\gamma$ is a curve passing through $\infty \in \hat{\C}$, use a Möbius transformation to send it to a bounded planar curve. 
Theorem \ref{T:Lambda-Proj-Inv} shows that this is well defined and the result now follows from the previous case.
\end{proof}

\subsection{Consequences of rigidity}\label{SS:rigidity}

Using a clever geometric description of their eponymous operator, Kerzman and Stein proved the following:
\begin{proposition}[\cite{KerSte78a}, Lemma 7.1]
Let $\Omega$ be a bounded simply connected planar domain with smooth boundary.
The Cauchy and Szeg\H{o} kernels coincide if and only if $\Omega$ is a disc.
\end{proposition}
In other words, $C(z,\zeta) = S(z,\zeta)$ for all $z \in \Omega$, $\zeta \in b\Omega$ if and only if $\Omega$ is a disc.

Theorem \ref{T:Lambda>=1} implies a much stronger rigidity theorem:

\begin{corollary}\label{C:vanishing-of-the-Kerzman-Stein-kernel}
Let $\Omega$ be a bounded simply connected planar domain with Lipschitz boundary.
If there exists a single $z \in \Omega$ such that the Cauchy and Szeg\H{o} kernels satisfy $|C(z,\zeta)| \le |S(z,\zeta)|$ for almost every $\zeta \in b\Omega$, then $\Omega$ is a disc.
\end{corollary}
\begin{proof}
Given $z \in \Omega$, consider the square of the $L^2$-distance $\norm{C(z,\cdot) - S(z,\cdot)}^2_{L^2(\gamma)} =$
\begin{align*}
&\int_\gamma |C(z,\zeta)|^2\,d\sigma(\zeta) + \int_\gamma |S(z,\zeta)|^2\,d\sigma(\zeta) -2 \re \int_\gamma C(z,\zeta) \ol{S(z,\zeta)} \,d\sigma(\zeta).
\end{align*}
Conjugate symmetry and the Szeg\H{o} reproducing property show that the second integral in the previous line evaluates to $S(z,z)$,
while the Cauchy integral formula yields
\begin{equation*}
2 \re \int_\gamma C(z,\zeta) \ol{S(z,\zeta)} \,d\sigma(\zeta) = 2 \re \int_\gamma C(z,\zeta) S(\zeta,z) \,d\sigma(\zeta) = 2\re S(z,z) = 2 S(z,z).
\end{equation*}
Since $\Omega$ is a bounded domain, Remark \ref{R:positivity-of-S(z,z)} says $S(z,z) > 0$.
Thus,
\begin{equation}\label{E:L2-distance-Szego-Cauchy}
\norm{C(z,\cdot) - S(z,\cdot)}^2_{L^2(\gamma)}
= \norm{C(z,\cdot)}_{L^2(\gamma)}^2 - S(z,z) =  
S(z,z) \left( \Lambda(b\Omega,z)^2 - 1 \right).
\end{equation}

If there exists a $z \in \Omega$ so that $C(z,\zeta) = S(z,\zeta)$ for almost every  $\zeta \in b\Omega$, then \eqref{E:L2-distance-Szego-Cauchy} = 0, meaning that $\Lambda(b\Omega,z) = 1$.
Theorem \ref{T:Lambda>=1} implies that $\Omega$ is a disc.
\end{proof}



\begin{corollary}\label{C:Upper-bound-on-capacity}
Let $\gamma$ be a simple closed Lipschitz curve in the plane.
Then its analytic capacity $\kappa(\gamma)$ and arc length $\sigma(\gamma)$ satisfy the following inequality:
\begin{equation}\label{E:Upper-bound-on-capacity}
\sigma(\gamma) \ge 2\pi \kappa(\gamma).
\end{equation}
Equality holds if and only if $\gamma$ is a circle.
\end{corollary}
\begin{proof}
Combining Theorems \ref{T:Lambda-at-infinity} and \ref{T:Lambda>=1}, we see that $\Lambda(\gamma,\infty) = \sqrt{\frac{\sigma(\gamma)}{2\pi \kappa(\gamma)}} \ge 1$, and that equality holds if and only if $\gamma$ is a circle.
\end{proof}

\begin{remark}
An estimate due to Ahlfors and Beurling gives a lower bound on the analytic capacity of a simple closed curve in terms of area enclosed (see, e.g., \cite[Chapter 5.3]{Ransford1995}):
Let $\gamma$ be a simple closed planar curve enclosing an area of $A(\gamma)$.
Then
\begin{equation}\label{E:Ahlfors-Beurling-ineq-2}
\kappa(\gamma) \ge \sqrt{A(\gamma)/\pi},
\end{equation}
with equality holding if and only if $\gamma$ is a circle.
Combining \eqref{E:Upper-bound-on-capacity} with \eqref{E:Ahlfors-Beurling-ineq-2} yields the isoperimetric inequality $\sigma(\gamma)^2 \ge 4\pi A(\gamma)$.
See \cite{GameKhav1989} for another proof of the isoperimetric inequality stemming from the Ahlfors-Beurling estimate.
\hfill $\lozenge$
\end{remark}


\section{The behavior of $\Lambda(\gamma,z)$ at the boundary}\label{S:Berezin-Transform}

Our goal here is to prove the following result, which confirms part \ref{II:Lambda-cont-at-curve} of Theorem \ref{T:intro-3-props-of-Lambda}.

\begin{theorem}\label{T:Lambda-continuinty-at-boundary}
Let $\gamma$ be a simple closed $C^1$ curve in the Riemann sphere.
Then the function $z \mapsto \Lambda(\gamma,z)$ is continuous on all of $\hat{\C}$.
In particular, if $\zeta_0 \in \gamma$, then $\lim_{z \to \zeta_0} \Lambda(\gamma,z) = 1$.
\end{theorem}

\subsection{Important kernel properties}\label{SS:Szego-blowup}
Let $X \subset \hat{\C}$ be a set and consider $f,g : X \to [0,\infty)$.
We say $f$ and $g$ are {\em comparable} on $X$ and write $f(z) \approx g(z)$, $z \in X$, if there exist constants $C_1,C_2>0$ such that for all $z \in X$,
\begin{equation*}
C_1 f(z) \le g(z) \le C_2 f(z).
\end{equation*}

\begin{proposition}\label{P:Szego-blow-up}
Let $\gamma$ be a simple closed Lipschitz curve in the complex plane and let $\delta(z)$ denote the distance of $z$ to $\gamma$. Then 
$
S_\pm(z,z) \approx \delta(z)^{-1},\;  z \in \Omega_+\cup\{z\in\Omega_-:\delta(z)<\ell\} .
$
\end{proposition}
\begin{proof}
Let $D(z,\delta)$ be the disc centered at $z$ of radius $\delta(z) = \delta$.
The Szeg\H{o} kernel of this disc is calculated using the unit disc formula \eqref{E:Szego-on-disc} and an appropriate affine map in the transformation law \eqref{E:SzegoTransLaw}.
From Szeg\H{o} kernel monotonicity in Proposition \ref{P:Szego-containment-monotonicity} we now obtain
\begin{equation}\label{E:E:Szego-upper-bound}
S_\pm(z,z) \le S_{D(z,\delta)}(z,z) = \frac{1}{2\pi\delta(z)}.
\end{equation}

For the other direction, consider first a point $z\in\Omega_+$ and the Riemann map $\Phi:\Omega_+\to\D$ with $\Phi(z)=0$, $\Phi'(z)>0$.  Using \eqref{E:Szego-on-disc} and \eqref{E:SzegoTransLaw} again we have $S_+(z,z)=\frac{\Phi'(z)}{2\pi}$.  Applying the (rescaled) Koebe one-quarter theorem \cite[Theorem 5.3.3]{{Ransford1995}} to $\Phi^{-1}$ we obtain
\begin{equation*}
\delta(z)\ge \frac14\left(\Phi^{-1}\right)'(0)=\frac{1}{4\,\Phi'(z)},
\end{equation*}
and so
\begin{equation}\label{E:E:Szego-lower-bound+}
 S_+(z,z)=\frac{\Phi'(z)}{2\pi} \ge \frac{1}{8\pi\delta(z)}.  
\end{equation}
Combining \eqref{E:E:Szego-upper-bound} with \eqref{E:E:Szego-lower-bound+} we have $S_+(z,z) \approx \delta(z)^{-1},\;  z \in \Omega_+.$

To treat $z\in\Omega_-$ close to $\gamma$, pick $z_0\in\Omega_+$ and $\ell>0$ so that the map $\eta(z):=\frac{1}{z-z_0}$ satisfies
\begin{itemize}
    \item $\eta$ is a bi-Lipschitz map from $U_\ell:=\{z\in\Omega_-:\delta(z)<\ell\}$ to $\eta\left(U_\ell\right)$
\end{itemize}
and
\begin{itemize}
    \item $|\eta'(z)|\approx 1$ on $U_\ell$.
\end{itemize}
Setting $\tilde\delta(w)$ to be the distance from $w$ to $\eta(\gamma)$ we have from our work above (with $\eta(\Omega_-)$ replacing $\Omega_+$) along the transformation law \eqref{E:SzegoTransLaw} that
\begin{align*}
 S_-(z,z) &\approx S_{\eta(\Omega_-)}  (\eta(z),\eta(z))\\
 &\approx \tilde\delta(\eta(z))\\
 &\approx \delta(z)
\end{align*}
for $z\in U_\ell$, completing the proof of the proposition.
\end{proof}

\begin{proposition}\label{P:kernel-adjoint-relationship}
Let $\bm{T}_1,\bm{T}_2$ be bounded projection operators from $L^2(\gamma)$ onto $\ch^2_\pm(\gamma)$, each represented by an integral kernel $K_j:\Omega_\pm \times \gamma \to \C$, such that for $f \in L^2(\gamma)$ and $z \in \Omega_\pm$,
\begin{equation*}
\bm{T}_j f(z) = \int_\gamma f(\zeta)K_j(z,\zeta)\,d\sigma(\zeta).
\end{equation*}
Additionally, assume $K_j(z,\cdot) \in L^2(\gamma)$ for $z \in \Omega_\pm$.
Then the following holds for a.e. $\zeta \in \gamma$:
\begin{equation}\label{E:kernel-adjoint-relationship}
\bm{T}_2^*\left( \ol{K_1(z,\cdot)} \right)(\zeta) = \ol{K_2(z,\zeta)}.
\end{equation}
\end{proposition}
\begin{proof}
Since $\bm{T}_2$ is bounded on $L^2(\gamma)$, there is a corresponding bounded adjoint $\bm{T}_2^*$.
By assumption, $\ol{K_1(z,\cdot)} \in L^2(\gamma)$ and so $\bm{T}^*_2\left( \ol{K_1(z,\cdot)} \right) \in L^2(\gamma)$.
Thus for $f \in L^2(\gamma)$,
\begin{equation}\label{E:Gen-Proj-1}
\left\langle f, \bm{T}^*_2\left( \ol{K_1(z,\cdot)} \right) \right\rangle = \left\langle \bm{T}_2(f),  \ol{K_1(z,\cdot)} \right\rangle 
= \bm{T}_1 \circ \bm{T}_2 f(z)
= \bm{T}_2 f(z),
\end{equation}
since $\bm{T}_2 f \in \ch^2_\pm(\gamma)$ and $\bm{T}_1$ is a projection onto $\ch_\pm^2(\gamma)$.
On the other hand,
\begin{equation}\label{E:Gen-Proj-2}
\left\langle f, \ol{K_2(z,\cdot)} \right\rangle = \bm{T}_2 f(z),
\end{equation}
by definition.
Equating \eqref{E:Gen-Proj-1} and \eqref{E:Gen-Proj-2} we see that $\bm{T}_2^*\left( \ol{K_1(z,\cdot)} \right) - \ol{K_2(z,\cdot)} \in L^2(\gamma)$ is orthogonal to all of $L^2(\gamma)$, and is therefore almost everywhere zero.
\end{proof}

\begin{corollary}
The Cauchy kernel and Szeg\H{o} kernels of $\ch^2_\pm(\gamma)$ are related as follows:
\begin{subequations}
\begin{alignat}{2}
&\bm{C}_\pm^*\left(S_\pm(\cdot,z)\right)(\zeta) = \ol{C_\pm(z,\zeta)}, \qquad &z \in \Omega_\pm,\quad \zeta \in \gamma, \label{E:Cauchy-adjoint-Szego} \\
&\bm{S}_\pm \left(\ol{C_\pm(z,\cdot)}\right)(\zeta) = S_\pm(\zeta,z), \qquad &z \in \Omega_\pm, \quad \zeta \in \gamma. \label{E:Szego-adjoint-Cauchy}
\end{alignat}
\end{subequations}
\end{corollary}
\begin{proof}
Apply Proposition \ref{P:kernel-adjoint-relationship} with $\bm{T}_1 = \bm{S}_\pm$ and $\bm{T}_2 = \bm{C}_\pm$ for \eqref{E:Cauchy-adjoint-Szego}.
Switch the roles of the operators and use the self-adjointness of the Szeg\H{o} projection to deduce \eqref{E:Szego-adjoint-Cauchy}.
\end{proof}

\subsection{The Berezin transform and the Kerzman-Stein operator}\label{SS:Lambda-at-boundary}


\subsubsection{The Berezin transform}\label{SSS:Berezin-trans}

Let $\gamma$ be a simple closed Lipschitz curve in the plane oriented counterclockwise. 
Given  $z \in \Omega_\pm\setminus \{\infty\}$, define the unit vector $s^\pm_z \in \ch_\pm^2(\gamma) \subset L^2(\gamma)$ by normalizing the Szeg\H{o} kernel as follows:
\begin{equation}\label{E:z-normalized-Szego}
s_z^\pm(\zeta) := \frac{S_\pm(\zeta,z)}{\sqrt{S_\pm(z,z)}}.
\end{equation}

\begin{lemma}\label{L:normalized-Szego-tends-weakly-to-0}
Let $\gamma$ be a simple closed Lipschitz curve in the plane and $z \in \C\backslash\{\gamma\}$.
The unit vectors $s^\pm_z \in \ch^2_\pm(\gamma)$ tend weakly to $0$ as $z$ approaches  $\gamma$.
\end{lemma}
\begin{proof} 
If $f \in L^2(\gamma)$ is perpendicular to $\ch_{\pm}^2(\gamma)$, observe that 
\begin{equation*}\langle f,s_z^\pm \rangle = S_\pm(z,z)^{-1/2} \bm{S}_\pm f(z) = 0.\end{equation*}
It is therefore sufficient to test only against functions $f$ in the Hardy space.

If $f \in \ch_{\pm}^2(\gamma)$ and $z \in \Omega_\pm \backslash \{\infty\}$, the Szeg\H{o} reproducing property gives 
\begin{equation}\label{E:tending-weakly-1}
\langle f, s_z^\pm \rangle = S_\pm(z,z)^{-1/2}f(z).
\end{equation}
By Remark \ref{R:density} the subspace $A^\alpha(\Omega_\pm) = \co(\Omega_\pm) \cap C^\alpha(\ol{\Omega_\pm})$ is dense in $\ch_\pm^2(\gamma)$, so we may choose a sequence of functions $\{f_n\} \subset A^\alpha(\Omega_\pm)$ tending to $f$ in the $L^2(\gamma)$-norm.
Then
\begin{align*}
|f(z) - f_n(z)| = \left| \int_\gamma S_\pm(z,\zeta)(f(\zeta) - f_n(\zeta))\,d\sigma(\zeta)  \right| \le S_\pm(z,z)^{1/2} \norm{f-f_n},
\end{align*}
which implies $|S_\pm(z,z)^{-1/2} f(z) - S_\pm(z,z)^{-1/2} f_n(z)| \le \norm{f-f_n}$.
Thus,
\begin{align*}
|S_\pm(z,z)^{-1/2}f(z)| &\le |S_\pm(z,z)^{-1/2}f(z) - S_\pm(z,z)^{-1/2}f_n(z)| + |S_\pm(z,z)^{-1/2}f_n(z)|  \notag \\
&\le \norm{f-f_n} + S_\pm(z,z)^{-1/2}|f_n(z)|.
\end{align*}

Now given $\epsilon >0$, we may choose $N$ large enough so that $\norm{f-f_N} < \frac{\epsilon}{2}$.
Since $f_N \in C^\alpha(\ol{\Omega_\pm})$, $|f_N|$ assumes a maximum value on the closure.  
And since $S_\pm(z,z) \approx \delta(z)^{-1}$ (Proposition \ref{P:Szego-blow-up}), $z$ can be taken sufficiently close to $\gamma$ to ensure that
\begin{equation*}
S_\pm(z,z)^{-1/2}|f_N(z)| \le S_\pm(z,z)^{-1/2} \,\sup|f_N| < \frac{\epsilon}{2}.
\end{equation*}
Now combining the above inequalities with \eqref{E:tending-weakly-1}, we have
\begin{equation*}
|\langle f,s_z^\pm \rangle| = |S_\pm(z,z)^{-1/2} f(z)| < \epsilon
\end{equation*}
for $z$ sufficiently close to $\gamma$.
Since $\epsilon$ was arbitrary, we conclude that $s_z^\pm$ tends weakly to 0 as $z$ is sent to $\gamma$.
\end{proof}

Suppose $\bm{T}$ is a bounded operator on $L^2(\gamma)$.
We define its {\em Berezin transform} to be the function $\wt{\bm{T}}:\Omega_\pm \backslash \{ \infty \} \to \C$ given by the formula
\begin{equation}\label{E:def-Berezin-transform-1}
\wt{\bm{T}}(z) := \big\langle \bm{T} s^\pm_z, s^\pm_z \big\rangle, \qquad z \in \Omega_\pm \backslash \{ \infty \}.
\end{equation}

The Berezin transform is important in the study of Toeplitz and Hankel operators in Bergman and Hardy space settings.
There is an extensive body of literature on this topic; see, e.g., the survey \cite{Zhu2021} and the references therein. 

\begin{lemma}\label{L:Berezin-weak-conv-to 0}
Let $\gamma$ be a simple closed Lipschitz curve and $\bm{T}$ a compact operator on $L^2(\gamma)$.
Then the Berezin transform $\wt{\bm{T}}(z)$ tends to $0$ as $z$ is sent to $\gamma$.
\end{lemma}
\begin{proof}
Since compact operators are completely continuous and $s_z^\pm$ tends weakly to 0 as $z$ is sent to $\gamma$, we have that $\norm{\bm{T} s_z^\pm} \to 0$. 
Now observe that
\begin{equation*}
\big|\wt{\bm{T}}(z)\big| = |\langle \bm{T} s_z^\pm,s_z^\pm \rangle| \le \norm{\bm{T} s_z^\pm} \norm{s_z^\pm} \le \norm{\bm{T} s_z^\pm},
\end{equation*}
which completes the proof.
\end{proof}

Now suppose that $\bm{T}_1$ and $\bm{T}_2$ are bounded operators on  $L^2(\gamma)$. 
We define a function $\sB(\gamma,\bm{T}_1,\bm{T}_2): \C\backslash\{\gamma\} \to \C$ by the formula
\begin{equation}\label{E:def-Berezin-transform-2}
\sB(\gamma,\bm{T}_1,\bm{T}_2)(z) 
:= \begin{dcases}
\big\langle \bm{T}_1 s^+_z, s^+_z \big\rangle,  & z \in \Omega_+, \\
\big\langle \bm{T}_2 s^-_z, s^-_z \big\rangle, & z \in \Omega_- \backslash \{\infty\}.
\end{dcases}
\end{equation}

We refer to $\sB(\gamma,\bm{T}_1,\bm{T}_2)$ as a {\em concatenated Berezin transform}.
This allows the consideration of two different operators on $\Omega_+$ and $\Omega_-$ simultaneously.

\subsubsection{Kerzman-Stein operators}\label{SSS:Kerzman-Stein}

Define the operator
\begin{equation}
\bm{A}_\pm = \bm{C}_\pm - \bm{C}_\pm^*.
\end{equation}
Kerzman and Stein \cite{KerSte78a} showed that the singularities of the Cauchy kernel and its adjoint cancel out as long as the associated curve is smooth.
Lanzani \cite{Lan99} improved the applicability of this result to $C^1$ curves.

\begin{proposition}\label{P:Kerzman-Stein-compact}
Let $\gamma$ be a $C^1$ curve in the complex plane. 
Then $\bm{A}_\pm$ is a compact operator on $L^2(\gamma)$.
\end{proposition}
\begin{proof}
See \cite{KerSte78a} for the original proof for $C^\infty$ domains and Lanzani's work \cite{Lan99} for the proof on $C^1$ curves.
Also see Bell's book \cite{Bell16} for a different perspective on the $C^\infty$ setting.
\end{proof}


\begin{proposition}\label{P:concat-Berezin-comps}
Let $\gamma$ be a simple closed Lipschitz curve in the complex plane.
The following computations hold for $z \in \C \backslash \{\gamma\}$:
\begin{subequations}
\begin{align}
&\sB(\gamma,\bm{A}_+,\bm{A}_-)(z) \equiv 0, \label{P:concat-Berezin-comps-A} \\
&\sB(\gamma,\bm{A}^2_+,\bm{A}^2_-)(z) = 1 - \Lambda(\gamma,z)^2, \label{P:concat-Berezin-comps-A^2}
\end{align}
where $\bm{A}_\pm^2 = \bm{A}_\pm \circ \bm{A}_\pm$.
\end{subequations}
\end{proposition}
\begin{proof}

Let $z \in \Omega_\pm \backslash \{\infty\}$.
For \eqref{P:concat-Berezin-comps-A}, we need only note that $\bm{C}_\pm$ fixes $s_z$.
Thus,
\begin{align*}
\sB(\gamma,\bm{A}_+,\bm{A}_-)(z) = \langle \bm{A}_\pm s_z^\pm,s_z^\pm \rangle &= \langle (\bm{C}_\pm - \bm{C}_\pm^*) s_z^\pm,s_z^\pm \rangle \\
&= \langle\bm{C}_\pm s_z^\pm,s_z^\pm \rangle - \langle s_z^\pm,\bm{C}_\pm s_z^\pm \rangle = \langle s_z^\pm,s_z^\pm \rangle - \langle s_z^\pm, s_z^\pm \rangle  = 0.
\end{align*}
For \eqref{P:concat-Berezin-comps-A^2}, we have that since both $\bm{C}_\pm$ and $\bm{C}_\pm^*$ are projections
\begin{align}
\sB(\gamma,\bm{A}^2_+,\bm{A}^2_-)(z)
&= \langle (\bm{C}_\pm - \bm{C}_\pm^*)^2 s_z^\pm,s_z^\pm \rangle \notag \\
&= \langle (\bm{C}^2_\pm - \bm{C}_\pm^*\bm{C}_\pm - \bm{C}_\pm\bm{C}_\pm^* + (\bm{C}_\pm^*)^2) s_z^\pm, s_z^\pm \rangle \notag \\
&= \langle\bm{C}_\pm s_z^\pm, s_z^\pm \rangle - \langle \bm{C}_\pm^* \bm{C}_\pm s_z^\pm, s_z^\pm \rangle - \langle\bm{C}_\pm \bm{C}_\pm^* s_z^\pm, s_z^\pm \rangle + \langle \bm{C}_\pm^* s_z^\pm, s_z^\pm \rangle \notag \\
&= \langle s_z^\pm, s_z^\pm \rangle - \langle \bm{C}_\pm s_z^\pm ,\bm{C}_\pm s_z^\pm \rangle - \langle\bm{C}_\pm \bm{C}_\pm^* s_z^\pm, s_z^\pm \rangle + \langle s_z^\pm, \bm{C}_\pm s_z^\pm \rangle \notag \\
&= 1 - \norm{\bm{C}_\pm^* s_z^\pm}^2. \label{E:Berezin-of-A^2}
\end{align}
But notice that
\begin{equation}\label{E:Lambda^2-as-a-norm}
\norm{\bm{C}_\pm^* s_z}^2 
= \frac{1}{S_\pm(z,z)} \norm{\bm{C}^*_\pm \big( S_\pm(\cdot,z) \big)}^2 
= \frac{1}{S_\pm(z,z)} \norm{\ol{C_\pm(z,\cdot)}}^2 
= \Lambda_\pm(\gamma,z)^2,
\end{equation}
where the second equality follows from \eqref{E:Cauchy-adjoint-Szego} and the last equality is by definition.
Combining \eqref{E:Berezin-of-A^2} and \eqref{E:Lambda^2-as-a-norm} gives the result.
\end{proof}



\begin{proof}[Proof of Theorem \ref{T:Lambda-continuinty-at-boundary}]
First assume that $\gamma$ is a simple closed $C^1$ curve in the plane.
It is clear from the definition that $z \mapsto \Lambda(\gamma,z)$ is continuous on $\C \backslash \{\gamma\}$, while continuity at $\infty$ has already been verified in Theorem \ref{T:Lambda-at-infinity}.
It thus remains to check what happens near the curve (by definition, $\Lambda(\gamma,\zeta_0) = 1$ for $\zeta_0 \in \gamma$).
Proposition \ref{P:Kerzman-Stein-compact} says $\bm{A}_\pm$ is compact, so $\bm{A}_\pm^2$ is also compact. 
Lemma \ref{L:Berezin-weak-conv-to 0} thus implies that 
\begin{equation*}
\sB(\gamma,\bm{A}^2_+,\bm{A}^2_-)(z) = \langle \bm{A}_\pm^2 s_z^\pm,s_z^\pm \rangle \to 0
\end{equation*}
as $z \in \Omega_\pm$ tends to any $\zeta_0 \in \gamma$.
But by Proposition \ref{P:concat-Berezin-comps}, this is equivalent to saying $\Lambda_\pm(\gamma,z)^2 \to 1$ as $z \in \Omega_\pm$ tends to $\zeta_0 \in \gamma$.
Since $\Lambda_\pm(\gamma,z)$ is positive, we conclude that $\lim_{z \to \zeta_0} \Lambda(\gamma,z) = 1 = \Lambda(\gamma,\zeta_0)$.

If $\gamma$ is a $C^1$ curve in the Riemann sphere passing through the point at infinity, then we use a Möbius transformation $\Phi$ to map it to a bounded $C^1$ curve $\Phi(\gamma)$.
Möbius invariance combined with the above argument now completes the proof.
\end{proof}


\section{Examples}\label{S:examples}

\subsection{Wedges}\label{SS:the-wedge}

Given $\theta \in (0,\pi)$, define two complementary wedges
\begin{subequations}
\begin{align}
&W_\theta = \{r e^{i\varphi} \in \C: r>0,\quad |\varphi| < \theta \}, \label{E:def-of-convex-wedge} \\
&V_\theta = \{r e^{i\varphi} \in \C: r>0,\quad \theta < \varphi < 2 \pi -\theta \}. \label{E:def-of-non-convex-wedge}
\end{align}
It suffices to consider $\theta \in (0,\frac{\pi}{2})$, so that $W_\theta$ is a convex set and $V_\theta$ is non-convex.
Let $\gamma_\theta$ parameterize the boundary $bW_\theta$:
\begin{equation}
\gamma_\theta(t) = 
\begin{dcases}
-t e^{i \theta}, &t \in (-\infty,0), \\
t e^{-i \theta}, &t \in [0,\infty), \\
\infty, & t = \infty.
\end{dcases}
\end{equation}
\end{subequations}

We now have a partition of the Riemann sphere $\hat{\C} = W_\theta \cup V_\theta \cup \gamma_\theta$.
In the notation of previous sections, we have $W_\theta = \Omega_+$ and  $V_\theta =\Omega_-$.
Thus we write 
\begin{subequations}
\begin{align}
&\Lambda_+(\gamma_\theta,r e^{i \varphi}) = \frac{1}{4\pi^2 S_{W_\theta}(r e^{i\varphi}, r e^{i\varphi})} \int_{\gamma_\theta} \frac{d\sigma(\zeta)}{|\zeta-r e^{i \varphi}|^2}, \qquad r>0, \quad \varphi \in (-\theta,\theta), \label{E:Lambda+wedge} \\
&\Lambda_-(\gamma_\theta,r e^{i \varphi}) = \frac{1}{4\pi^2 S_{V_\theta}(r e^{i\varphi}, r e^{i\varphi})} \int_{\gamma_\theta} \frac{d\sigma(\zeta)}{|\zeta-r e^{i \varphi}|^2}, \qquad r>0, \quad \varphi \in (\theta,2\pi - \theta). \label{E:Lambda-wedge}
\end{align}
\end{subequations}

\subsubsection{Szeg\H{o} kernels}

The Szeg\H{o} kernels of $W_\theta$ and $V_\theta$, can be computed from their Riemann maps.
For $\alpha \in (0,\pi)$, let $W_\alpha$ be the (possibly non-convex) wedge given by \eqref{E:def-of-convex-wedge}.
It is straightforward to verify that the Riemann map $\Psi_\alpha: W_\alpha \to \D$ takes the form
\begin{equation}
\Psi_\alpha(z) = \frac{1-z^{\frac{\pi}{2\alpha}}}{1+z^{\frac{\pi}{2\alpha}}},
\end{equation}
where the fractional power $z^{\frac{\pi}{2\alpha}}$ refers to the branch preserving the positive real axis.

Setting $\alpha = \theta$ and $z = r e^{i\varphi}$, the transformation law in Proposition~\ref{P:SzegoTransLaw} and \eqref{E:Szego-on-disc} show
\begin{subequations}
\begin{equation}\label{E:Szego-kernel-Wtheta}
S_{W_\theta}(re^{i\varphi},re^{i\varphi}) 
= \frac{1}{8r\theta}\sec\left(\frac{\pi \varphi}{2\theta}\right), \qquad r>0, \quad \varphi \in (-\theta,\theta).
\end{equation}

The Szeg\H{o} kernel for $V_\theta$ is computed similarly.
First observe that the map $z \mapsto -z$ sends $V_\theta$ to $W_{\pi-\theta}$.
From here, apply the map $\Psi_{\pi-\theta}$ to obtain the Riemann map from $V_\theta$ to $\D$.
\begin{equation}\label{E:Szego-kernel-Vtheta}
S_{V_\theta}(re^{i\varphi},re^{i\varphi}) 
= \frac{1}{8r(\pi-\theta)}\sec\left(\frac{\pi}{2}\frac{(\pi - \varphi)}{(\pi - \theta)}\right), \qquad r>0, \quad \varphi \in (\theta,2\pi-\theta).
\end{equation}
\end{subequations}

\subsubsection{$L^2$-norm of the Cauchy kernel}

Computation of the integrals in \eqref{E:Lambda+wedge} and \eqref{E:Lambda-wedge} is assisted by the following 
\begin{lemma}\label{L:Wedge-aux-integral-calculation}
Let $\alpha \in (0,2\pi)$ and $r>0$.
Then
\begin{equation}\label{E:Wedge-aux-integral-calculation}
\ci(r,\alpha) := \int_0^\infty \frac{dx}{|x-re^{i\alpha}|^2} =
\frac{1}{r\sinc(\pi-\alpha)},
\end{equation}
where 
\begin{equation*} \sinc(t):=
\begin{dcases}
\frac{\sin(t)}{t}, & t \neq 0 \\
1, & t=0.
\end{dcases}
\end{equation*}
\end{lemma}
\begin{proof}
If $\alpha = \pi$, the fundamental theorem of calculus gives the result.
When $\alpha \in (0,\pi)$,
\begin{align}
\ci(r,\alpha) 
= \int_0^\infty \frac{dx}{x^2+r^2-2xr\cos\alpha} 
&= \frac{1}{r^2\sin^2\alpha}\int_0^\infty \frac{dx}{1+ \left(\frac{x-r\cos\alpha}{r\sin\alpha}\right)^2} \notag \\
&= \frac{1}{r\sin\alpha}\int_{-\!\cot\alpha}^\infty \frac{du}{1+u^2} \notag \\
&= \frac{1}{r\sin\alpha} \left( \frac{\pi}{2} + \arctan(\cot \alpha) \right). \label{E:arctan-of-cotan}
\end{align}
Elementary trigonometry now confirms that \eqref{E:Wedge-aux-integral-calculation} holds in this case.
For $\alpha\in (\pi,2\pi)$, reflection across the horizontal axis reveals that $\ci(r,\alpha) = \ci(r,2\pi-\alpha)$.
Combining this with the earlier result for $\alpha \in (0,\pi]$ shows that \eqref{E:Wedge-aux-integral-calculation} holds for $\alpha \in (0,2 \pi)$.
\end{proof}

Fix $\theta \in (0,\tfrac{\pi}{2})$ and take $z = r e^{i\varphi} \in W_\theta$, with $r>0$, $\varphi \in (-\theta,\theta)$.
Using Lemma \ref{L:Wedge-aux-integral-calculation} it is easily verified that (just draw a picture)
\begin{subequations}
\begin{align}
\norm{C(re^{i\varphi},\cdot)}^2_{L^2(\gamma_\theta)} 
&= \frac{1}{4\pi^2}\left( \ci(r,\theta-\varphi) + \ci(r,\theta+\varphi) \right)  \notag\\
&= \frac{1}{4\pi^2 r}\left( \frac{1}{\sinc(\pi-(\theta-\varphi))} + \frac{1}{\sinc(\pi-(\theta+\varphi))} \right). \label{E:norm-Cauchy-kernel-Wtheta}
\end{align}

Similarly, take $z = r e^{i\varphi} \in V_\theta$, with $r>0$ and $\varphi \in (\theta,2\pi-\theta)$.
Then Lemma \ref{L:Wedge-aux-integral-calculation} gives
\begin{align}
\norm{C(re^{i\varphi},\cdot)}^2_{L^2(\gamma_\theta)} 
&= \frac{1}{4\pi^2}\left( \ci(r,\varphi - \theta) + \ci(r,2\pi - \theta - \varphi) \right) \notag \\
&=\frac{1}{4\pi^2 r}\left( \frac{1}{\sinc(\pi-(\varphi-\theta))} + \frac{1}{\sinc(\pi-(\varphi+\theta))} \right). \label{E:norm-Cauchy-kernel-Vtheta}
\end{align}
\end{subequations}

\subsubsection{The $\Lambda$-function}
From Theorem \ref{T:Lambda-Proj-Inv} we easily see $\Lambda(\gamma_\theta,r e^{i\varphi})$ is independent of $r>0$.  Alternately, by canceling factors of $r^{-1}$ in the quotients of \eqref{E:norm-Cauchy-kernel-Wtheta} and \eqref{E:norm-Cauchy-kernel-Vtheta}   by the corresponding results from  \eqref{E:Szego-kernel-Wtheta}  and \eqref{E:Szego-kernel-Vtheta}  we obtain the following

\begin{theorem}\label{T:Lambda-wedge-formula}
For $\theta \in (0,\frac{\pi}{2})$,
the function $re^{i\varphi} \mapsto \Lambda(\gamma_\theta,r e^{i\varphi})^2$ is computed on $W_\theta = \Omega_+$ and $V_\theta = \Omega_-$.
\begin{itemize}[wide]
\item[$(a)$] Let $z = r e^{i \theta} \in W_\theta$, so that $r>0$ and $\varphi \in (-\theta,\theta)$. Then
\begin{align*}
\Lambda_+(\gamma_\theta,re^{i\varphi})^2
= \frac{2\theta}{\pi^2} \left( \frac{1}{\sinc(\pi-(\theta-\varphi))} + \frac{1}{\sinc(\pi-(\theta+\varphi))} \right) \cos\left(\frac{\pi \varphi}{2\theta}\right). 
\end{align*}

\item[$(b)$] Let $z = r e^{i \theta} \in V_\theta$, so that $r>0$ and $\varphi \in (\theta, 2\pi-\theta)$. Then
\begin{align*}
\Lambda_-(\gamma_\theta,re^{i\varphi})^2
=\frac{2(\pi-\theta)}{\pi^2} \left( \frac{1}{\sinc(\pi-(\varphi-\theta))} + \frac{1}{\sinc(\pi-(\varphi+\theta))} \right) \cos\left(\frac{\pi}{2}\frac{(\pi - \varphi)}{(\pi - \theta)}\right).
\end{align*}
\end{itemize}
\end{theorem}

\subsubsection{Remarks on the formula}
Since $\Lambda(\gamma_\theta,r e^{i \varphi})$ is independent of $r>0$, let us define 
\begin{equation*}
L(\theta,\varphi) := \Lambda(\gamma_\theta, e^{i \varphi}).
\end{equation*}
Theorem \ref{T:Cauchy-norm-lower-bound} tells us that $\norm{\bm{C}}_{L^2(\gamma_\theta)} \ge \sup \{ L(\theta,\varphi) : \varphi \in [-\theta, 2\pi-\theta) \}$.

\noindent $\bullet$ In \cite{Bolt07b} Bolt observes that a Möbius transformation maps the wedge $W_\theta$  onto a lens with vertices at $\pm 1$.
This lens has boundary length $\sigma(\gamma) = 4\theta \csc\theta$ and capacity $\kappa(\gamma) = \pi/(2(\pi-\theta))$, and Bolt uses this information to obtain a lower bound on $\norm{\bm{A}}$; see Remark \ref{R:Bolt-lower-bound}.
The closely related lower bound on $\norm{\bm{C}}$ given by $\Lambda(\gamma,\infty)$ in Theorem \ref{T:Lambda-at-infinity} together with the Möbius invariance of $\Lambda$ shows 
\begin{equation*}
\norm{\bm{C}}_{L^2(\gamma_\theta)} 
\ge \sqrt{\frac{\sigma(\gamma)}{2\pi\kappa(\gamma)}}
=\frac{2}{\pi} \sqrt{(\pi-\theta)\theta\csc\theta} := B(\theta).
\end{equation*}

\noindent $\bullet$ The graphs of $L(\theta,\varphi)$ and $B(\theta)$ for $\theta = \tfrac{\pi}{8}, \tfrac{\pi}{4}$ are displayed in Figure \ref{Im:plots-of-Lambda-on-wedge}. 
$B(\theta)$ agrees with the maximum value (at $\varphi = 0$) of $L(\theta,\varphi)$ when the angles correspond to the interior domain $W_\theta$, but is strictly less than the maximum value when $\varphi$ is taken from the exterior $V_\theta$.

\begin{figure}[htp]
\centering
\includegraphics[width=.45\textwidth]{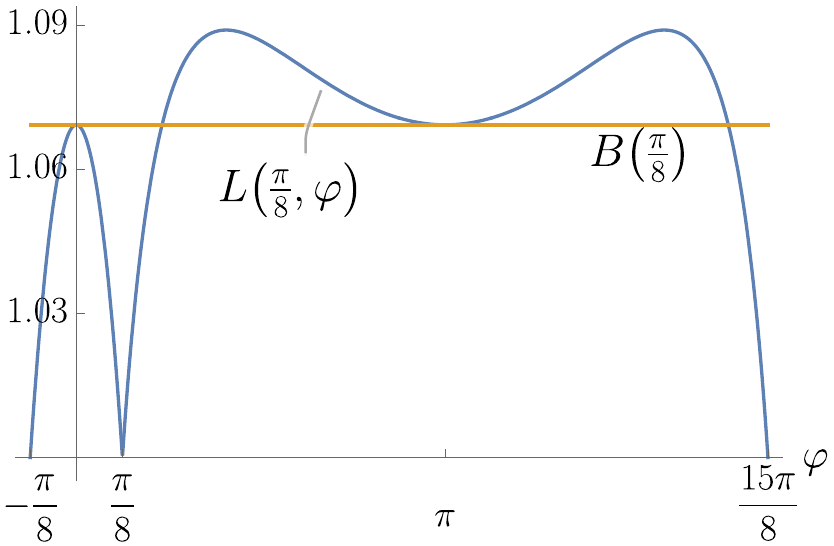} \quad 
\includegraphics[width=.45\textwidth]{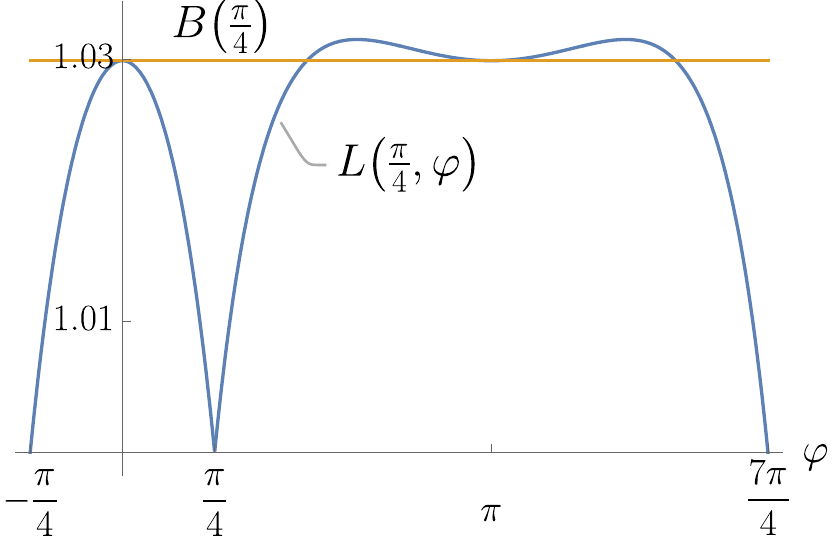}

\caption{Behavior of $\varphi \mapsto L(\theta,
\varphi)$ for $\theta = \frac{\pi}{8}$ and $\theta = \frac{\pi}{4}$.}
\label{Im:plots-of-Lambda-on-wedge}
\end{figure}

\noindent $\bullet$ When $z \in W_\theta \cup V_\theta$, Theorem \ref{T:Lambda>=1} guarantees that $\Lambda(\gamma_\theta,z)>1$.
But the formulas in Theorem \ref{T:Lambda-wedge-formula} show that $\Lambda(\gamma_0,z)$ $\to 1$ as $z=re^{i\varphi}$ tends to any smooth point on the curve $\gamma_\theta$ (meaning that $r>0$ and $\varphi \to -\theta$, $\theta$ or $2\pi -\theta$).
This is illustrated in Figure \ref{Im:plots-of-Lambda-on-wedge} when $\theta= \tfrac{\pi}{8}, \tfrac{\pi}{4}$, and should be compared with the continuity result given in Theorem \ref{T:Lambda-continuinty-at-boundary}.

\noindent $\bullet$ Since the non-smooth boundary points $0, \infty \in \gamma_\theta$ can be approached from any angle $\varphi \in (-\theta,2\pi-\theta)$, the function
$re^{i\varphi} \mapsto \Lambda(\gamma_\theta, re^{i\varphi})$ is discontinuous at both.
In light of the relationship between $\Lambda$, the Berezin transform and the Kerzman-Stein operator $\bm{A}$ in Section \ref{SSS:Kerzman-Stein}, lack of continuity shows that $\bm{A}$ is not compact.
Bolt and Raich \cite{BolRai15} show that $\bm{A}$ is never compact when corners are present.


\subsection{Ellipses}\label{SS:the-ellipse}

The next family of curves we consider are ellipses.
For $r \ge 1$, define
\begin{equation}\label{E:def-of-ellipse}
\mathcal{E}_r = \Big\{x+iy \in \C : \frac{x^2}{r^2}+y^2 = 1 \Big\}.
\end{equation}
As usual $\ce_r$ is oriented counterclockwise, so that $\Omega^r_+$ is the filled in ellipse.

\subsubsection{Families of special functions}\label{SSS:families-of-special-functions}

Elliptic integrals, elliptic functions and Jacobi theta functions comprise a deep and beautiful area of mathematics.
In what follows, the reader is not assumed to have prior background and all necessary definitions are given. 
References to properties relevant to our computations of $\Lambda(\ce_r,z)$ are also interspersed as needed.

The elliptic integrals of the first, second and third kinds ($K,E,\Pi$, respectively) make up a canonical set to which all other elliptic integrals can be reduced.
In what follows, variables $k \in (0,1)$ and $n \in \R$ are called the {\em elliptic modulus} and {\em characteristic}, respectively.
We follow conventions used by Whittaker and Watson \cite[Chapter 22.7]{WhitWatBook}, but the reader is cautioned that other conventions are also in common use (especially in mathematical software; e.g., Mathematica implements these as functions of $k^2$, not $k$):
\begin{subequations}\label{E:3-Elliptic-integrals}
\begin{align}
K(k) &= \int_0^1 \frac{dt}{\sqrt{(1-t^2)(1-k^2t^2)}}, \label{E:def-elliptic-int-K} \\
E(k) &= \int_0^1 \sqrt{\frac{1-k^2t^2}{1-t^2}}\,dt, \label{E:def-elliptic-int-E} \\
\Pi(n,k) &= \int_0^1 \frac{dt}{(1-n t^2)\sqrt{(1-t^2)(1-k^2t^2)}}. \label{E:def-elliptic-int-Pi}
\end{align}
\end{subequations}

Next, recall the Jacobi theta functions, where $z \in \C$ and $|q|<1$ (see \cite[Chapter 21]{WhitWatBook}):
\begin{subequations}\label{E:3-Elliptic-functions}
\begin{align}
&\vartheta_1(z,q) = 2 q^\frac{1}{4} \sum_{j=0}^\infty (-1)^j q^{j(j+1)} \sin[(2j+1)z], \label{E:Jacobi-elliptic-theta-1} \\
&\vartheta_2(z,q) = 2 q^\frac{1}{4} \sum_{j=0}^\infty q^{j(j+1)} \cos[(2j+1)z], \label{E:Jacobi-elliptic-theta-2}\\
&\vartheta_3(z,q) = 1 + 2\sum_{j=1}^\infty q^{j^2} \cos(2jz), \label{E:Jacobi-elliptic-theta-3} \\
&\vartheta_4(z,q) = 1 + 2\sum_{j=1}^\infty (-1)^j q^{j^2} \cos(2jz). \label{E:Jacobi-elliptic-theta-4}
\end{align}
\end{subequations}
Theta functions also admit elegant infinite product expansions; see \cite[Chapter 21.3]{WhitWatBook}.

Define for $k \in [0,1]$, the elliptic nome $q(k)$ by
\begin{subequations}
\begin{equation}\label{E:def-nome}
q(k) = \exp{\bigg[\frac{-\pi K(\sqrt{1-k^2})}{K(k)}\bigg]}.
\end{equation}
This function is strictly increasing with $q(0)=0$ and $q(1)=1$.
The inverse nome $k(q)$ is defined for $q \in [0,1)$ by an infinite product, or equivalently by a ratio of theta functions:
\begin{equation}\label{E:inverse-nome}
k(q) = 4\sqrt{q}\prod_{j=1}^\infty \left[ \frac{1+q^{2j}}{1+q^{2j-1}} \right]^4 =  \frac{\vartheta_2(0,q)^2}{\vartheta_3(0,q)^2}.
\end{equation}
\end{subequations}
A slight abuse of notation lets us write $k(q(k)) = k$ and $q(k(q))=q$ on $[0,1)$; 
also $\lim_{q \to 1} k(q) = 1$, though the infinite product/theta function formula is not valid at $q=1$.

Finally, define Jacobi's elliptic $\mathrm{sn}$ function in terms of the above functions
\begin{equation}\label{E:def-Jacobi-sn}
\mathrm{sn}(u,k) = \frac{\vartheta_3(0,q(k))}{\vartheta_2(0,q(k))} \cdot \frac{\vartheta_1\bigg(\dfrac{u}{\vartheta_3(0,q(k))^2},q(k)\bigg)}{\vartheta_4\bigg(\dfrac{u}{\vartheta_3(0,q(k))^2},q(k)\bigg)}.
\end{equation}
The function $u \mapsto \mathrm{sn}(u,k)$ is is meromorphic and doubly periodic with quarter periods $K := K(k)$ and $iK' := iK(\sqrt{1-k^2})$, i.e., $\mathrm{sn}(u,k) = \mathrm{sn}(u + 4K,k)$ and $\mathrm{sn}(u,k) = \mathrm{sn}(u + 4iK',k)$.


\subsubsection{The norm of the Cauchy kernel}

For $z \in \Omega^r_+ \cup \Omega^r_-$, we write the norm $\norm{C_\pm(z,\cdot)}_{L^2(\ce_r)}$ as an integral over $[-1,1]$.
First parametrize the top and bottom halves of $\ce_r$ by
\begin{equation}\label{E:Ellipse-Param}
\zeta^\pm(t) = rt \pm i\sqrt{1 - t^2}, \qquad -1\le t \le 1.
\end{equation}
Writing the arc length differential $d\sigma(\zeta)$ in terms of this parametrization gives 
\begin{equation}\label{E:arclength-differential}
d\sigma(\zeta) = |d\zeta(t)| =  r\,\sqrt{\frac{1-(1-\frac{1}{r^2})t^2}{1-t^2}}\,dt.
\end{equation}
If $z = \alpha+i\beta$, with $\alpha,\beta \in \R$, then \eqref{E:Ellipse-Param} implies
\begin{align*}
|\zeta^\pm(t) - z|^2 &= |(rt-\alpha) \pm i\big(\sqrt{1-t^2} \mp \beta\big)|^2 \notag \\
&= \alpha^2 + \beta^2 +1 - 2r\alpha t + (r^2-1)t^2 \mp 2\beta\sqrt{1-t^2}.
\end{align*}
Now combine this with \eqref{E:arclength-differential} to see
\begin{align}
\norm{C(z,\cdot)}^2 &=\frac{1}{4\pi^2} \int_{\ce_r} \frac{d\sigma(\zeta)}{|\zeta - z|^2} \notag \\
&= \frac{1}{4\pi^2} \int_{-1}^1 \frac{|d\zeta(t)|}{|\zeta^+(t) - \alpha-i\beta|^2} +  \frac{1}{4\pi^2}  \int_{-1}^1 \frac{|d\zeta(t)|}{|\zeta^-(t) - \alpha-i\beta|^2} \notag \\ 
&= \frac{r}{2\pi^2}\int_{-1}^1 \frac{\alpha^2 + \beta^2 + 1 - 2r\alpha t + (r^2-1)t^2}{(\alpha^2 + \beta^2 + 1 - 2r\alpha t + (r^2-1)t^2)^2 - 4\beta^2(1-t^2)} \sqrt{\frac{1-(1-\frac{1}{r^2})t^2}{1-t^2}}\,dt. \label{E:norm-comp-gen-z-1}
\end{align}

In the special case in which $\alpha = \beta = 0$, \eqref{E:norm-comp-gen-z-1} reduces to
\begin{align}
\norm{C(0,\cdot)}^2_{L^2(\ce_r)}
&= \frac{1}{2\pi^2r}\int_{-1}^1 \frac{r^2  - (r^2-1)t^2}{1  + (r^2-1)t^2} \frac{1}{\sqrt{(1-t^2)(1-(1-\frac{1}{r^2})t^2)}}\,dt \notag \\
&= \frac{1}{\pi^2r}\int_{0}^1 \left( \frac{r^2+1}{1 - (1-r^2)t^2} -1 \right) \frac{1}{\sqrt{(1-t^2)(1-(1-\frac{1}{r^2})t^2)}}\,dt \notag \\
&= \frac{1}{\pi^2 r} \Big( (r^2+1) \cdot \Pi\Big(1-r^2, \sqrt{1-\tfrac{1}{r^2}}\Big) - K\Big(\sqrt{1-\tfrac{1}{r^2}}\Big) \Big). \label{E:Cauchy-norm-at-0}
\end{align}

\subsubsection{The Szeg\H{o} kernel of the interior domain}

The Riemann map of the ellipse $\Theta_r : \Omega_+^r \to \D,$ takes the following form (see, e.g., \cite[Chapter VI]{Nehari1952book} or \cite{SzegoAMM1950}):
\begin{equation}\label{E:ellipse-Riemann-map-1}
\Theta_r(z) = \sqrt{k_r}\cdot \mathrm{sn}\left(\frac{2 K(k_r)}{\pi} \arcsin{\bigg(\frac{z}{\sqrt{r^2-1}}\bigg)} \, , \, k_r \right).
\end{equation}
Here $\mathrm{sn(\cdot,\cdot)}$ is the elliptic function \eqref{E:def-Jacobi-sn}, $K$ is the elliptic integral \eqref{E:def-elliptic-int-K}.
The elliptic modulus $k_r \in [0,1)$ is the unique value (determined by \eqref{E:inverse-nome}) satisfying
\begin{equation}\label{E:inverse-elliptic-modulus-1}
q(k_r) = \bigg(\frac{r-1}{r+1}\bigg)^2.
\end{equation}
The value of $k_r$ in \eqref{E:ellipse-Riemann-map-1} is determined by the eccentricity of $\ce_r$, while the factor $\sqrt{r^2-1}$ appearing in the arcsine accounts for the fact that the foci of $\ce_r$ are at $(\pm(r^2-1),0)$.

From \cite[Example 21.6.5]{WhitWatBook}, it can be seen that
\begin{equation}\label{E:two-theta-identities}
\frac{2 K(k)}{\pi} = \vartheta_3\big(0,q(k)\big)^2, \qquad 
\sqrt{k} = \frac{\vartheta_2(0,q(k))}{\vartheta_3(0,q(k))}.
\end{equation}

Equations \eqref{E:inverse-elliptic-modulus-1}, \eqref{E:two-theta-identities} and the definition of $\mathrm{sn}(\cdot,\cdot)$ in \eqref{E:def-Jacobi-sn} lets us rewrite \eqref{E:ellipse-Riemann-map-1}:
\begin{align}
\Theta_r(z)
&= \frac{\vartheta_2\big(0,\big(\tfrac{r-1}{r+1}\big)^2 \big)}{\vartheta_3\big(0,\big(\tfrac{r-1}{r+1}\big)^2 \big)} \cdot \mathrm{sn}\left( \vartheta_3 \bigg(0,\bigg(\dfrac{r-1}{r+1} \bigg)^2\bigg)^2  \arcsin{\bigg(\frac{z}{\sqrt{r^2-1}}\bigg)} \, , \, k_r \right) \notag \\
&= \frac{\vartheta_1\left(\arcsin{\bigg(\dfrac{z}{\sqrt{r^2-1}}\bigg)},\bigg(\dfrac{r-1}{r+1}\bigg)^2 \right)}{\vartheta_4\left(\arcsin{\bigg(\dfrac{z}{\sqrt{r^2-1}}\bigg)},\bigg(\dfrac{r-1}{r+1}\bigg)^2 \right)}. \label{E:Theta_r-form-2}
\end{align}

The Szeg\H{o} transformation formula \eqref{P:SzegoTransLaw} gives $S_+(z,z)$ in terms of $\Theta_r$ and $\Theta_r'$
\begin{align}\label{E:Szego-interior-formula}
S_+(z,z) = \frac{|\Theta_r'(z)|}{2\pi(1-|\Theta_r(z)|^2)}.
\end{align}
When $z=0$, this formula simplifies.
Indeed, \eqref{E:Theta_r-form-2} shows that $\Theta_r(0) = 0$ and (letting $\vartheta'_1$ denote differentiation in the first slot), the quotient rule gives
\begin{equation}\label{E:Theta'-1}
\Theta_r'(0) = \frac{\vartheta'_1\left( 0, \big(\tfrac{r-1}{r+1}\big)^2 \right)}{ \sqrt{r^2-1} \cdot  \vartheta_4\left( 0, \big(\tfrac{r-1}{r+1}\big)^2 \right)}.
\end{equation}

In \cite[Section 21.4]{WhitWatBook}, Whittaker and Watson establish the following {\em ``remarkable result"} of Jacobi, saying {\em ``several proofs have been given, but none are simple"}:
\begin{equation*}
\vartheta'_1(0,q) = \vartheta_2(0,q) \vartheta_3(0,q) \vartheta_4(0,q).
\end{equation*}
This can be inserted into \eqref{E:Theta'-1} to show
\begin{equation}\label{E:Szego-at-0-ellipse}
S_+(0,0) = \frac{|\Theta'_r(0)|}{2\pi} = \frac{\vartheta_2\left( 0, \big(\tfrac{r-1}{r+1}\big)^2 \right) \vartheta_3\left( 0, \big(\tfrac{r-1}{r+1}\big)^2 \right)}{2\pi \sqrt{r^2-1}}.
\end{equation}


\subsubsection{Calculation of $\Lambda(\ce_r,0)$ and $\Lambda(\ce_r,\infty)$}\label{SSS:Calculation-Lambda-ellipse}

\begin{theorem}\label{T:Cauchy-ellipse-lower-bound}
Let $r \ge 1$.
The function $z \mapsto \Lambda(\ce_r,z)^2$ assumes the following values.
\begin{subequations}
\begin{align}
&\Lambda(\ce_r,\infty)^2 = \frac{4r}{\pi(r+1)} E\Big(\sqrt{1-\tfrac{1}{r^2}}\Big), \label{E:Cauchy-ellipse-lower-bound-infty} \\
&\Lambda(\ce_r,0)^2 = \frac{2}{\pi} \sqrt{1-\tfrac{1}{r^2}} \cdot \frac{ (r^2+1) \cdot \Pi\Big(1-r^2, \sqrt{1-\tfrac{1}{r^2}}\Big) - K\Big(\sqrt{1-\tfrac{1}{r^2}}\Big) }{\vartheta_2\Big(0,\big(\frac{r-1}{r+1}\big)^2 \Big) \, \vartheta_3\Big(0,\big(\frac{r-1}{r+1}\big)^2 \Big)}. \label{E:Cauchy-ellipse-lower-bound-0}
\end{align}
\end{subequations}
\end{theorem}
\begin{proof}
The $z=\infty$ formula follows from the definition of $\Lambda$ in \eqref{E:Lambda-def}.
(Recall that Theorem \ref{T:Lambda-at-infinity} shows that $\Lambda(\ce_r,z)$ is continuous at $z=\infty$.)
It is known (see \cite[Table 5.1]{Ransford1995}) that the capacity $\kappa$ of the ellipse $\{x^2/a^2 + y^2/b^2 = 1 \}$ is $\frac{a+b}{2}$, so $\kappa(\ce_r) = \frac{r+1}{2}$.
On the other hand, \eqref{E:arclength-differential} shows the arc length of $\ce_r$ is given by
\begin{equation*}
\sigma(\ce_r) = \int_{\ce_r}d\sigma(\zeta) = 2r\int_{-1}^1 \sqrt{\frac{1-(1-\frac{1}{r^2})t^2}{1-t^2}}\,dt = 4r E\Big(\sqrt{1-\tfrac{1}{r^2}}\Big).
\end{equation*}

For the $z=0$ formula, simply divide \eqref{E:Cauchy-norm-at-0}  by \eqref{E:Szego-at-0-ellipse}.
\end{proof}

\subsubsection{Remarks}
We discuss properties of $\Lambda(\ce_r,0)$ and $\Lambda(\ce_r,\infty)$ as $r$ varies, giving special attention to the endpoint cases of $r \to 1$ and $\infty$.
Both values give ``asymptotically sharp" estimates on $\norm{\bm{C}}$ as $r\to 1$, though $\Lambda(\ce_r,0)$ is larger and thus ``sharper" (see Figure \ref{Im:plots-of-Lambda-on-ellipse_antipodal_points}).
We also briefly discuss $\Lambda(\ce_r,z)$ for other $z$ values.

\begin{figure}[htp]
\centering
\includegraphics[width=.95\textwidth]{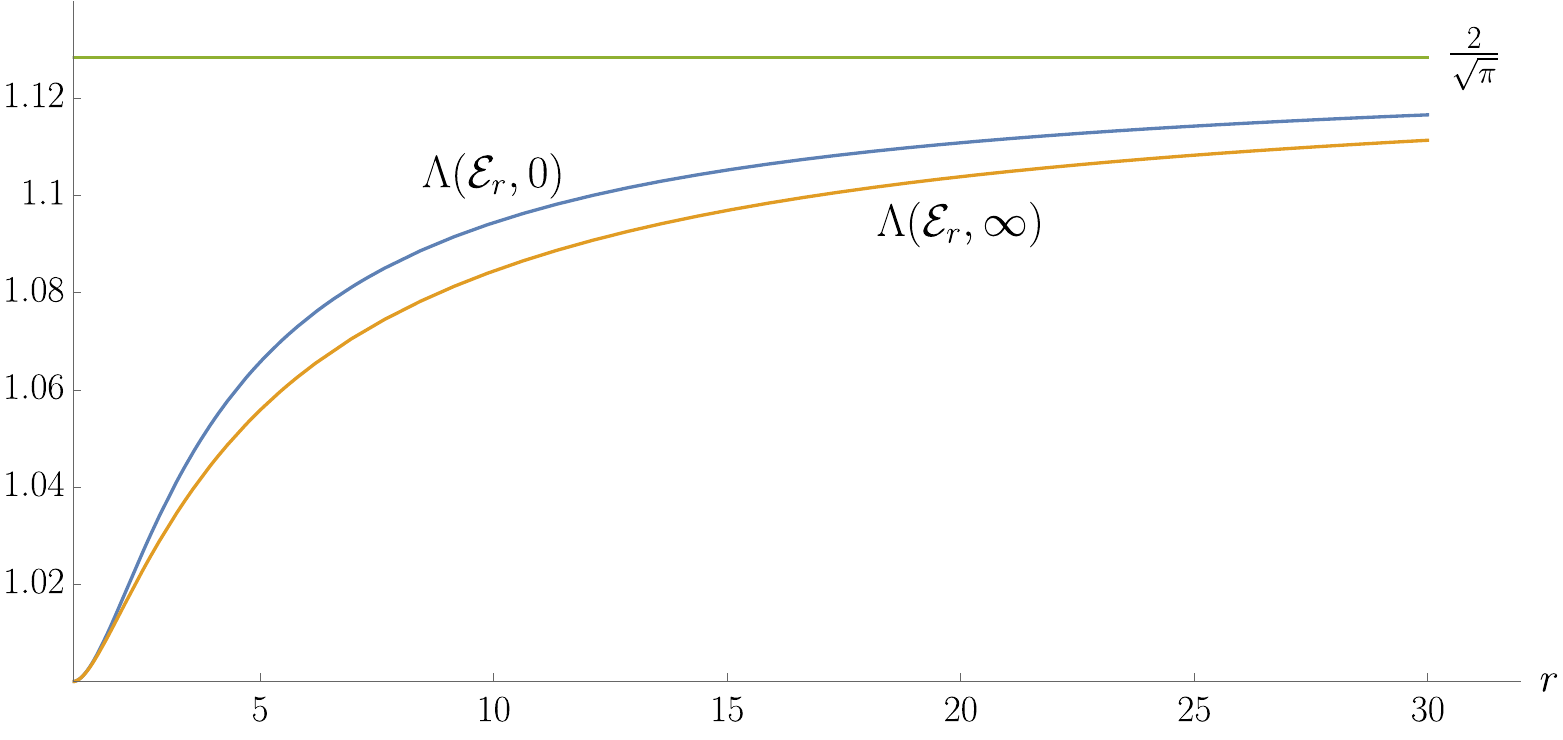}
\caption{Comparison of $r\mapsto \Lambda(\ce_r,0)$ and $r\mapsto \Lambda(\ce_r,\infty)$ for $r\ge1$.}
\label{Im:plots-of-Lambda-on-ellipse_antipodal_points}
\end{figure}

\noindent $\bullet$ {\em Behavior of $\Lambda(\ce_r,\infty)$ and $\Lambda(\ce_r,0)$.}
Using well known asymptotic behavior of elliptic integrals and theta functions (see \cite{HandbookElliptic71}) we expand $\Lambda(\ce_r,\infty)$ and $\Lambda(\ce_r,0)$ near $r=1$:
\begin{subequations}
\begin{align}
\Lambda(\ce_r,\infty) &= 1 + \tfrac{1}{32}(r-1)^2 - \tfrac{1}{32}(r-1)^3 + \tfrac{3}{128}(r-1)^4 + O(|r-1|^5) \label{E:infinfty-expansion} \\
\Lambda(\ce_r,0) &= 1 + \tfrac{1}{32}(r-1)^2 - \tfrac{1}{32}(r-1)^3 + \tfrac{7}{200}(r-1)^4 + O(|r-1|^5). \label{E:zero-expansion}
\end{align}
\end{subequations}
(For the reader interested in working out these details by hand, it is convenient to re-write $\Lambda(\ce_r,0)$ using the so-called {\em Heuman Lambda function}; see \cite[page 225]{HandbookElliptic71}).

In \cite{Bol07}, Bolt shows that the spectrum of the Kerzman-Stein operator $\bm{A}$ on an ellipse consists of eigenvalues $\pm i \lambda_l$, where each $\pm i \lambda_l$ has multiplicity $2$ and $\lambda_1 \ge \lambda_2 \ge \cdots \ge 0$.
He then provides asymptotics of the eigenvalues as the eccentricity tends to zero.
The largest number on the list is $\lambda_1 = \norm{\bm{A}} = \sqrt{||\bm{C}||^2-1}$, and we deduce from Bolt's estimates that, as $r \to 1$,
\begin{equation}\label{E:Bolt-asymptotic}
\norm{\bm{C}} \approx \sqrt{ 1 + \tfrac{1}{4} \big(\tfrac{r-1}{r+1}\big)^2 } = 1 + \tfrac{1}{32}(r-1)^2 
+ O(|r-1|^3).
\end{equation}
Comparing \eqref{E:Bolt-asymptotic} to \eqref{E:infinfty-expansion} and \eqref{E:zero-expansion}, we see the expansions for $\Lambda(\ce_r,\infty)$ and $\Lambda(\ce_r,0)$ are both asymptotically sharp as $r \to 1$.
But if we then compare the coefficients of $(r-1)^4$, we see that $\Lambda(\ce_r,0)$ is the better lower bound on $\norm{\bm{C}}$ near $r = 1$ (since $\tfrac{7}{200} > \frac{3}{128}$).
This information proves Theorem \ref{T:Cauchy-ellipse-lower-bound-intro}.

Known asymptotic expansions of elliptic integrals and theta functions also yield the behavior of $\Lambda(\ce_r,\infty)$ and $\Lambda(\ce_r,0)$ as $r \to \infty$:
\begin{equation*}
\lim_{r \to \infty} \Lambda(\ce_r,0) = \frac{2}{\sqrt{\pi}} = \lim_{r \to \infty} \Lambda(\ce_r,\infty).
\end{equation*}

\begin{remark}
The Mathematica generated plot in Figure \ref{Im:plots-of-Lambda-on-ellipse_antipodal_points} suggests $\Lambda(\ce_r,0) > \Lambda(\ce_r,\infty)$ for all $r \ge 1$, and it would be interesting to prove this.
It also appears in the plot that both functions are strictly increasing in $r$.
This is relatively straightforward to prove in the case of $\Lambda(\ce_r,\infty)$, but the $\Lambda(\ce_r,0)$ case seems to be harder.
\hfill $\lozenge$
\end{remark}

\begin{remark}
The above information should be considered together with a known upper bound on the norm of $\bm{C}$.
Adapting results by Feldman, Krupnik and Spitkovsky in \cite{FelKruSpi1996}, we see that for $r \ge 1$,
\begin{equation}\label{E:FKS-Cauchy-ellipse-upper-bound}
\norm{\bm{C}}_{L^2(\ce_r)} \le \sqrt{1 + \big(\tfrac{r-1}{r+1}\big)^2}.
\end{equation}
In particular, the norm of the Cauchy transform on any ellipse is always less than $\sqrt{2}$.
\hfill $\lozenge$
\end{remark}

\begin{figure}[htp]
\centering
\includegraphics[width=.95\textwidth]{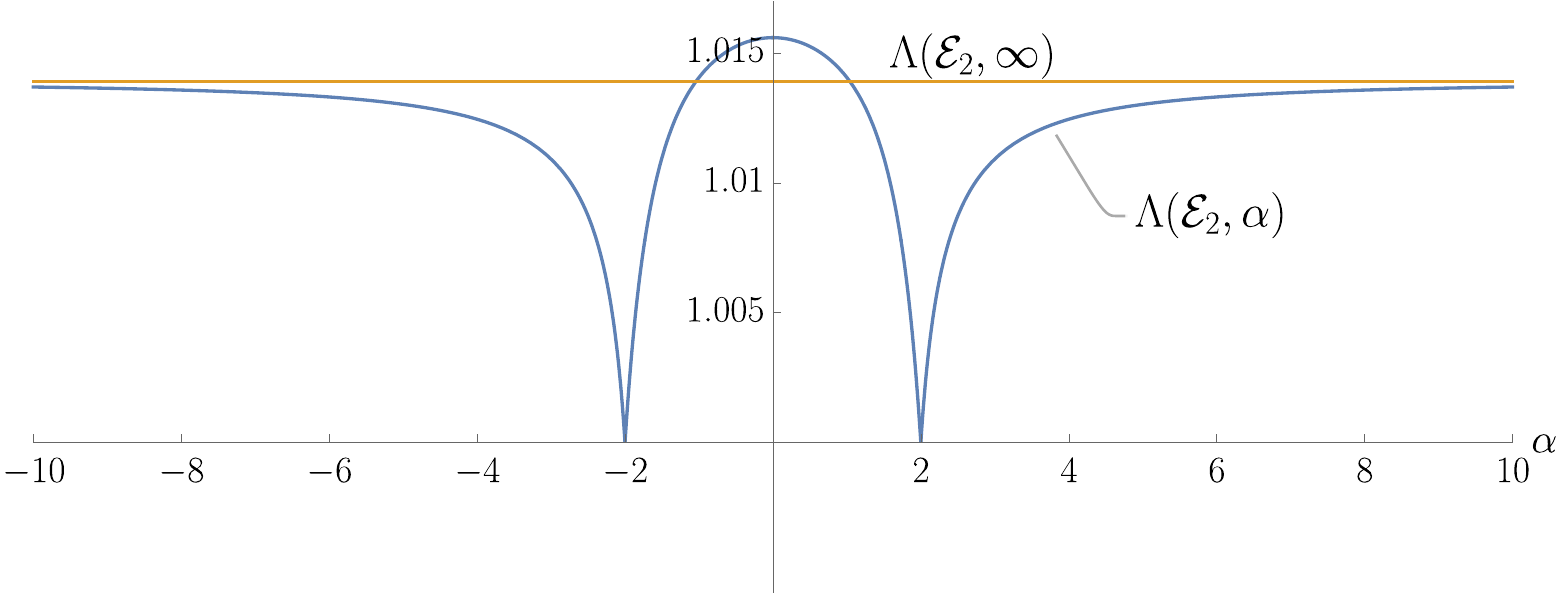}
\caption{For $\alpha \in \R$, plot of $\alpha \mapsto \Lambda(\ce_2, \alpha)$ shows maximum at $\alpha=0$.}
\label{Im:plots-of-Lambda-on-ellipse_r=2}
\end{figure}

\noindent $\bullet$ {\em Values of $\Lambda(\ce_r,z)$ for $z\neq 0,\infty$.}
The formulas provided in \eqref{E:norm-comp-gen-z-1}, \eqref{E:Theta_r-form-2} and \eqref{E:Szego-interior-formula} are valid for $z \in \Omega_+^r$.
Numerical evidence for specific $r$ values suggests that $z=0$ may in fact maximize $\Lambda_+(\ce_r,z)$.
This is illustrated in Figure \ref{Im:plots-of-Lambda-on-ellipse_r=2}, when $r=2$ and $z = \alpha$ is real valued.
In this picture, the interior domain corresponds to $\alpha \in (-2,2)$.

To compute $\Lambda_-(\ce_r,z)$, we need the exterior Riemann map.
The map from the unit disc $\D$ to $\Omega_-^r$ (the complement of the solid ellipse) is given by a Joukowski map (see \cite[page 270]{Nehari1952book}).
Such maps can be inverted explicitly and the desired $\Psi_r: \Omega_-^r \to \D$ obtained.
The particular map used to generate the figure (the $r=2$ case) for $|\alpha| > 2$ is given by
\begin{equation*}
\Psi_2(z) = \frac{3}{z+\sqrt{z^2-3}},
\end{equation*}
from which the exterior Szeg\H{o} kernel $S_-(z,z)$ can be obtained.
This is then combined with the Cauchy norm computation in \eqref{E:norm-comp-gen-z-1} to yield the exterior $\Lambda$-function.

\bibliographystyle{acm}
\bibliography{BarEdh21}
 
\end{document}